\documentclass{article}
\usepackage{amsmath,amssymb,amsthm}
\usepackage{graphicx,subfigure,float,url,color}
\usepackage{mathrsfs}
\usepackage[colorlinks=true]{hyperref}
\usepackage{pdfsync}

\usepackage{bm}
\usepackage{dsfont}
\def\R{\mathbb R}
\def\N{\mathbb N}

\def\trait (#1) (#2) (#3){\vrule width #1pt height #2pt depth #3pt}
\def\fin{\hfill\trait (0.1) (5) (0) \trait (5) (0.1) (0) \kern-5pt
\trait (5) (5) (-4.9) \trait (0.1) (5) (0)}

\numberwithin{equation}{section}

\newcommand{\be}{\begin{equation}}
\newcommand{\ee}{\end{equation}}
\newcommand{\baa}{\begin{array}}
\newcommand{\eaa}{\end{array}}
\newcommand{\ba}{\begin{eqnarray}}
\newcommand{\ea}{\end{eqnarray}}

\newcommand{\ban}{\begin{eqnarray*}}
\newcommand{\ean}{\end{eqnarray*}}

\newcommand{\dis}{\displaystyle}
\newtheorem{theorem}{\bf Theorem}[section]
\newtheorem{lemma}[theorem]{\bf Lemma}
\newtheorem{proposition}[theorem]{\bf Proposition}

\theoremstyle{definition}\newtheorem{remark}[theorem]{\bf Remark}

\newcommand{\ds}{\,\mathrm{d}s}

\def\apg{\left\{}
\def\chg{\right\}}
\def\apt{\left(}
\def\cht{\right)}

\def\ch{\right.}

\def\a{a}

\renewcommand{\H}{\mathcal{H}}

\newcommand{\psg}[2]{\big\langle #1,#2\big\rangle}

\renewcommand{\leq}{\leqslant}

\renewcommand{\H}{\mathcal{H}}

\def\ds{\rightarrow}

\newcommand{\vsx}{\mathbb{X}}
\newcommand{\vsz}{\mathbb{Z}}

\newcommand{\vcr}{\mathbb{V}}

\newcommand{\co}{C}
\newcommand{\cor}{\mathcal{C}}
\newcommand{\V}{S}
\newcommand{\Vr}{\Sigma}
\newcommand{\Sm}{\mathcal{M}}
\newcommand{\xo}{x^0}
\newcommand{\xf}{x^f}
\newcommand{\too}{t^0}
\newcommand{\tf}{t^f}
\newcommand{\To}{T_0}
\newcommand{\Tu}{T_1}
\newcommand{\Zo}{Z_0}
\newcommand{\Zu}{Z_1}

\newcommand{\Zob}{\bar{Z}_0}
\newcommand{\Zub}{\bar{Z}_1}

\newcommand{\rou}{\rho_1}
\newcommand{\rod}{\rho_2}
\newcommand{\roh}{\rho_\H}

\newcommand{\wu}{w_1}
\newcommand{\wdu}{w_2}
\newcommand{\wh}{w_\H}

\newcommand{\vu}{v_1}
\newcommand{\vd}{v_2}
\newcommand{\vh}{v_\H}

\newcommand{\xud}{\chi}
\newcommand{\xof}{\lambda}
\newcommand{\xofo}{\xof^{0}}
\newcommand{\xoff}{\xof^f}
\newcommand{\vapz}{\mathbb{P}_\vsz}
\newcommand{\vayu}{P_{Y_1}}
\newcommand{\vayd}{P_{Y_2}}
\newcommand{\vayh}{P_{Y_\H}}
\newcommand{\varu}{P_{\rou}}
\newcommand{\vard}{P_{\rod}}
\newcommand{\varh}{P_{\roh}}

\newcommand{\vaxu}{P_{1}}
\newcommand{\vaxd}{P_{2}}
\newcommand{\vaxh}{P_{\H}}

\newcommand{\pt}{\partial} 
\newcommand{\ps}{\nabla}

\newcommand{\Hrs}{\widetilde{\mathcal{H}}}
\newcommand{\Hr}{{\mathcal{H}}}

\newcommand{\Hus}{\widetilde{\textit{H}_1}}
\newcommand{\Hu}{{\textit{H}_1}}
\newcommand{\Hds}{\widetilde{\textit{H}_2}}
\newcommand{\Hd}{{\textit{H}_2}}
\newcommand{\Hhs}{\widetilde{\textit{H}_\H}}
\newcommand{\Hh}{{\textit{H}_\H}}

\newcommand{\hysh}{{\rm (H $1\H2$)}}

\newcommand{\psgH}[2]{\big\langle #1,#2\big\rangle_{\H}}

\date{}

\title{Value function for regional control problems via dynamic programming and Pontryagin maximum principle}

\author{G. Barles, A. Briani
\thanks{Laboratoire de Math\'ematiques et Physique Th\'eorique (UMR CNRS 7350), F\'ed\'eration Denis Poisson (FR CNRS 2964), Universit\'e Fran\c{c}ois Rabelais, Parc de Grandmont, 37200 Tours, France (\texttt{Guy.Barles@lmpt.univ-tours.fr}, \texttt{ariela.briani@lmpt.univ-tours.fr})
\newline \indent This work was partially supported by the ANR HJnet ANR-12-BS01-0008-01},  
E. Tr\'elat\thanks{Sorbonne Universit\'es, UPMC Univ Paris 06, CNRS UMR 7598, Laboratoire Jacques-Louis Lions, F-75005, Paris, France (\texttt{emmanuel.trelat@upmc.fr}).}}

\begin{document}
\maketitle

\begin{abstract}
In this paper we focus on regional deterministic optimal control problems, i.e., problems where the dynamics and the cost functional may be different in several regions of the state space and  present discontinuities at their interface. 

Under the assumption that optimal trajectories have a locally finite number of switchings (no Zeno phenomenon), we use the duplication technique to show that the value function of the regional optimal control problem is the minimum over all possible structures of trajectories of value functions associated with classical optimal control problems settled over fixed structures, each of them being the restriction to some submanifold of the value function of a classical optimal control problem in higher dimension.
The lifting duplication technique is thus seen as a kind of desingularization of the value function of the regional optimal control problem. In turn, we extend to regional optimal control problems the classical sensitivity relations and we prove that the regularity of this value function is the same (i.e., is not more degenerate) than the one of the higher-dimensional classical optimal control problem that lifts the problem.
\end{abstract}

\\
{\bf AMS Class. No}:
49L20,   
49K15,   
35F21. 


\section{Introduction} 
In this article, we consider regional optimal control problems in finite dimension, the word ``regional" meaning that the dynamics and the cost functional may depend on the region
of the state space and therefore present discontinuities at the interface between these different regions. Our objective is to provide a description of these trajectories exploiting the Pontryagin maximum principle and the Dynamic Programming approach (the value function is the viscosity solution of the corresponding Hamilton-Jacobi equation). We establish a relationship between these two approaches, which is new for regional control problems.

There is a wide existing literature on regional optimal control problems, which have been  studied with different approaches and within various related contexts: \textit{stratified optimal control problems} in \cite{BC,BrYu,HZ}, \textit{optimal multiprocesses} in \cite{ClarkeVinter1,ClarkeVinter2}, they also enter into the wider class of \textit{hybrid optimal control} (see \cite{BBM,RIK,SC}).
Necessary optimality conditions have been developed in \cite{GP,HT,Sussmann} in the form of a Pontryagin maximum principle. For regional optimal control problems, the main feature is the jump of the adjoint vector at the interface between two regions (see \cite{HT}).
An alternative approach is the Bellman one, developed in \cite{BBC,BBC2,RZ} in terms of an appropriate Hamilton-Jacobi equation studied whose solutions are studied in the viscosity sense (see also \cite{IMZ,Ou,RSZ} for transmission conditions at the interface).

In this paper we exploit both the Dynamic Programming approach and Pontryagin maximum principle in order to describe the optimal trajectories of regional control problems. Although the techniques are not new we believe that the approach is interesting and helpful.
We are going to use in an instrumental way the lifting duplication technique, nicely used in \cite{Dmitruk} in order to prove that the hybrid version of the Pontryagin maximum principle can be derived from the classical version (i.e., for classical, non-hybrid problems) under the assumption that optimal trajectories are regular enough. More precisely, we assume that optimal trajectories have a locally finite number of switchings, or, in other words, we assume that wild oscillation phenomena (known as Fuller, Robbins or Zeno phenomena in the existing literature, see \cite{CGPT} for a survey) do not occur, or at least, if they happen then we deliberately ignore the corresponding wildy oscillating optimal trajectories and we restrict our search of optimal trajectories to those that have a regular enough structure, i.e., a locally finite number of switchings. Under this assumption, the duplication technique developed in \cite{Dmitruk} can be carried out and shows that the regional optimal control problem can be lifted to a higher-dimensional optimal control problem that is ``classical", i.e., non-regional. As we are going to see, this construction has a number of nice applications.

\medskip

In order to point out the main ideas, we consider the following simplified framework with only two different regions.
Let $N\in\N^*$. We assume that
\begin{equation*}
\begin{split}
& \R^N=\Omega_1\cup\Omega_2\cup\H,\qquad \Omega_1,\Omega_2\ \textrm{open}, \quad \Omega_1\cap\Omega_2=\emptyset,\\
& \H=\partial \Omega_1=\partial\Omega_2 \ \textrm{is a $C^1$-submanifold of $\R^N$,}
\end{split}
\end{equation*}
and we consider a nonlinear optimal control problem in $\R^N$, stratified according to the above partition. We write this regional optimal control problem as
\begin{equation}\label{regionalOCP}
\begin{split}
& \dot \vsx(t) = f(\vsx(t),\a(t)), \\
& \vsx(t^0)=x^0,\quad \vsx(t^f)=x^f, \\
& \inf \int_{t^0}^{t^f} \ell(\vsx(t),\a(t))\, dt,
\end{split}
\end{equation}
where the dynamics $f$ and the running cost $\ell$ are defined as follows.
If $x\in\Omega_i$ for $i=1$ or $2$ then
$$
f(x,\a) = f_i(x,\a),\qquad \ell(x,\a) = l_i(x,\a),
$$
where $f_i:\R^N\times\R^m\rightarrow\R^N$ and $l_i:\R^N\times\R^m\rightarrow\R$ are $C^1$-mappings. If $x\in\H$ then
$$
f(x,\a) = f_\H(x,\a),\qquad \ell(x,\a) = l_\H(x,\a),
$$
where $f_\H:\R^N\times\R^m\rightarrow\R^N$ and $\ell_\H:\R^N\times\R^m\rightarrow\R$ are $C^1$-mappings. 
The set $\H$ is called the {\it interface} between the two open regions $\Omega_1$ and $\Omega_2$ (see Figures \ref{fig_1} and  \ref{fig_2}).

The class of controls that we consider also depends on the region. As long as $\vsx(t)\in\Omega_i$, we assume that $\a \in L^\infty ((t^0,t^f),A_i)$, where $A_i$ is a measurable subset of $\R^m$. Accordingly, as long as $\vsx(t)\in\H$, we assume that $\a \in L^\infty ((t^0,t^f),A_\H)$,  $A_\H$ is a measurable subset of $\R^m$.

The terminal times $t^0$ and $t^f$ and the terminal points $x^0$ and $x^f$ may be fixed or free according to the problem under consideration.
For instance, if we fix $x^0, t^0, x^f, t^f$, we define the value function 
$$
\V(x^0,t^0,x^f,t^f)
$$
of the regional optimal control problem \eqref{regionalOCP} as being the infimum of the cost functional over all possible admissible trajectories steering the control system from $(x^0,t^0)$ to $(x^f,t^f)$.

Our objective is to show that the value function $\V$ of the regional optimal control problem \eqref{regionalOCP} can be recovered from the study of a {\it classical} (i.e., non-hybrid) optimal control problem settled in high dimension, under the assumption of finiteness of switchings. To this aim, we list all possible structures of optimal trajectories of \eqref{regionalOCP}.
We recall that, for regional optimal control problems, existence of an optimal control and Cauchy uniqueness results are derived using Filippov-like arguments, allowing one to tackle the discontinuities of the dynamics and of the cost functional (see, e.g., \cite{BC,BrYu,HZ}). 

In what follows, we assume that the regional optimal control problem under consideration admits at least one optimal solution. We consider such an optimal trajectory $\vsx(\cdot)$ associated with a control $\a(\cdot)$ on $[t^0,t^f]$. Assuming that $x^0\in\Omega_1$ and $\xf \in\Omega_2$, we  consider various structures. 

\medskip

The simplest case is when the trajectory $\vsx(\cdot)$ consists of two arcs, denoted by $([t^0,t^1],X_1(\cdot))$ and $([t^1,t^f],X_2(\cdot))$, lying respectively in $\Omega_1$ for the first part, and then in  $\Omega_2$ for the second part of the trajectory, with $X_1(t^1)=X_2(t^1)\in\H$. Such optimal trajectories are studied in \cite{HT} under the assumption of a transversal crossing and an explicit jump condition is given for the adjoint vector obtained by applying the Pontryagin maximum principle. This is the simplest possible trajectory structure, and we denote it by 1-2 (see Figure \ref{fig_1}). It has only one switching.

\begin{figure}[h]
\begin{center}
\resizebox{8cm}{!}{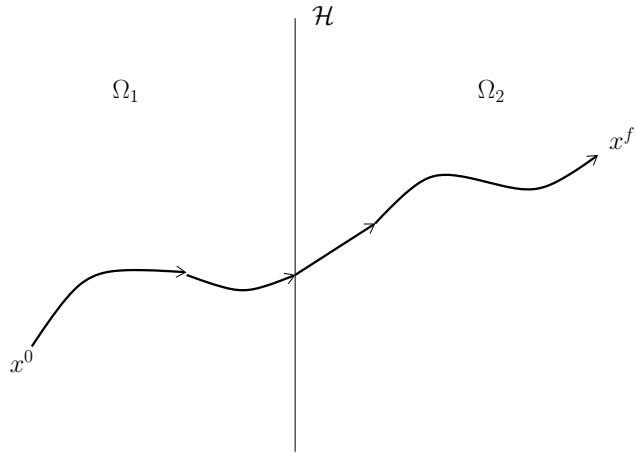}
\end{center}
\caption{Structure 1-2.}\label{fig_1}
\end{figure}

The second structure is when the trajectory $\vsx(\cdot)$ consists of three arcs, denoted by $([t^0,t^1],X_1(\cdot))$, $([t^1,t^2],X_\H(\cdot))$ and $([t^2,t^f],X_2(\cdot))$, lying respectively in $\Omega_1$ for the first arc, in $\H$ for the second arc and in $\Omega_2$ for the third arc. The middle arc $X_\H$ lies along the interface. Such a structure is denoted by 1-$\H$-2 (see Figure \ref{fig_2}). The trajectory has two switchings.

\begin{figure}[h]
\begin{center}
{ \resizebox{8cm}{!}{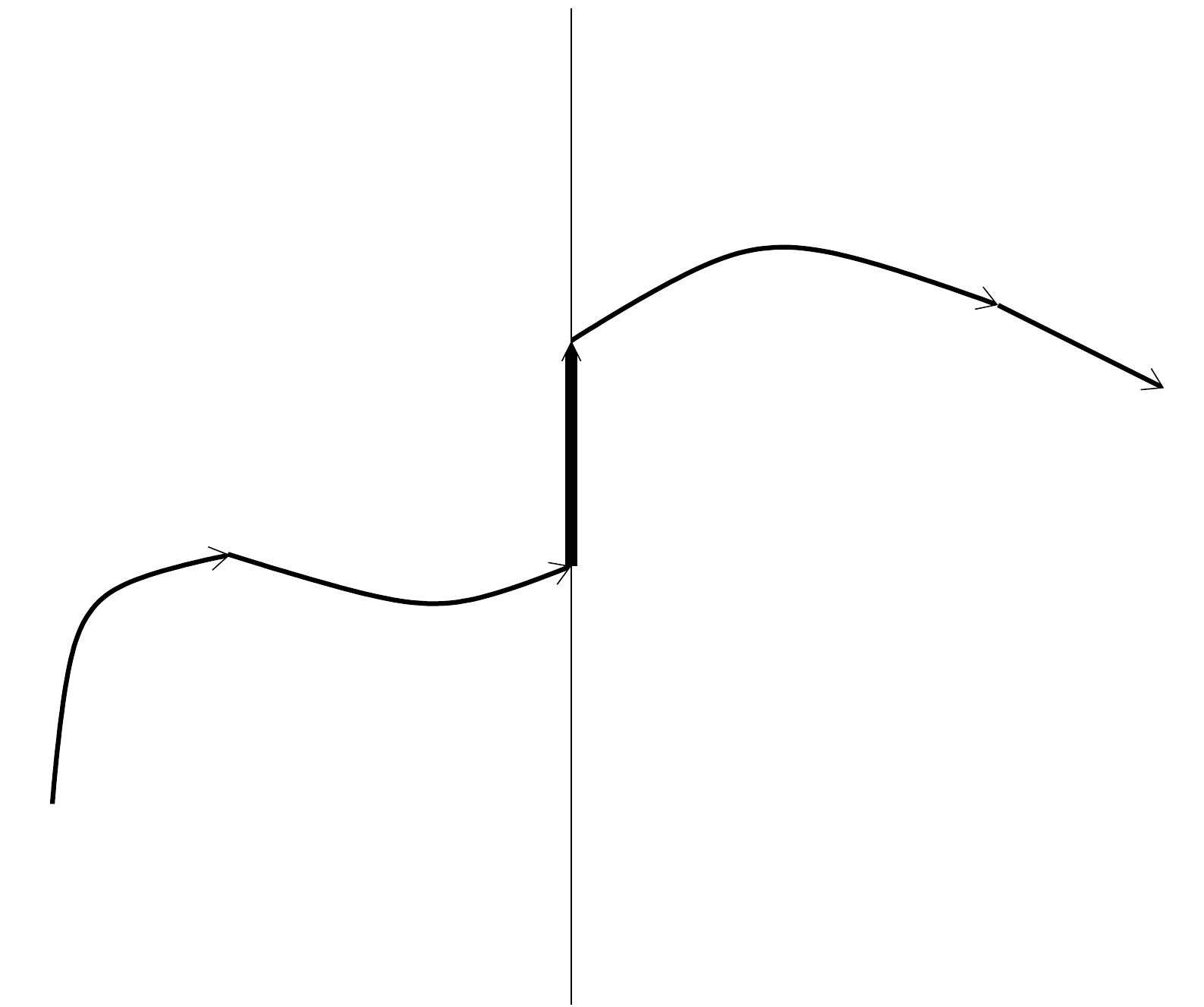} }
\end{center}
\caption{Structure 1-$\H$-2.}\label{fig_2}
\end{figure}

Accordingly, we consider all possible structures 1-2-$\H$-1, 1-$\H$-1-2, 1-2-$\H$-2, etc, made of a finite number of successive arcs. Restricting ourselves to any such fixed structure, we can define a specific optimal control problem consisting of finding an optimal trajectory steering the system from the initial point to the desired target point and minimizing the cost functional over all admissible trajectories having exactly such a structure. Denoting by $\V_{12}$, $\V_{1\H 2}$, etc, the corresponding value functions, we have
$$
\V=\inf \{ \V_{12}, \V_{1\H 2},\ldots \} ,
$$
provided all optimal trajectories of the regional optimal control problem have a locally finite number of switchings (and thus, the infimum above runs over a finite number of possibilites).

Using the duplication argument of \cite{Dmitruk}, we show that each of the above value functions (restricted to some fixed structure) can be written as the projection / restriction of the value function of a \emph{classical} optimal control problem in higher dimension (say $p$, which is equal to the double of the number of switchings of the corresponding structure), the projection being considered along some coordinates, and the restriction being done to some submanifolds of the higher dimensional space $\R^p$. The word ``duplication" reflects the fact that each arc of the trajectory gives two components of the dynamics of the problem in higher dimension. 

Thanks to this technique, we characterize the value function as a viscosity solution of an Hamilton-Jacobi equation and we apply the classical Pontryagin maximum principle. We thus provide an explicit relationship between the gradient of the value function of the regional control problem evaluated along the optimal trajectory and the adjoint vector. This sensitivity relation extends to the framework of regional optimal control problems the relation in the classical framework.
This allows us to derive conditions at the interface: continuity of the Hamiltonian and jump condition for the adjoint vector.

In Section \ref{sec2} we provide the details of the procedure for the structures 1-2 and 1-$\H$-2. The procedure goes similarly for other structures and consists of designing a duplicated problem of dimension two times the number of arcs of the structure.

The value function $\V$ is then the infimum of value functions associated with all possible structures, provided optimal trajectories have a locally finite number of switchings. The latter assumption is required to apply the duplication technique.
However in general it may happen that the structure of switchings have a complex structure, even fractal, and thus the set of switching points may be countably or even uncountably infinite. In the context of hybrid optimal control problems, the \emph{Zeno phenomenon} is a well known chattering phenomenon, meaning that the control switches an infinite number of times over a compact interval of times. It is analyzed for instance in \cite{JELS,ZJLS}, and necessary and/or sufficient conditions for the occurrence of the Zeno phenomenon are provided in \cite{AAS,HLF}. However, we are not aware of any existing result providing sufficient conditions for hybrid optimal control problems under which the number of switchings of optimal trajectories is locally finite or even only countable. Anyway, although the Zeno phenomenon may occur in general, restricting the search of optimal strategies to trajectories having only a locally finite number of switchings is a reasonable assumption in practice in particular in view of numerical implementation (see \cite{CGPT,TJOTA}).

Under this local finiteness assumption, it follows from our analysis that the regularity of the value function $\V$ of the regional optimal control problem is the same (i.e., is not more degenerate) than the one of the higher-dimensional classical optimal control problem lifting the problem. More precisely, we prove that each value function $\V_{12}$, $\V_{1\H 2}$, $\ldots$, for each fixed structure, is the restriction to a submanifold of the value function of a classical optimal control problem in higher dimension. Our main result, Theorem \ref{mainthm}, gives a precise representation of the value function and of the corresponding sensitivity relations, in relation with the adjoint vector coming from the Pontryagin maximum principle.
In particular, if for instance all classical value functions above are Lipschitz then the value function of the regional optimal control problem is Lipschitz as well.
This regularity result is new in the framework of hybrid or regional optimal control problems.

\medskip

The paper is organized as follows. 

In Section \ref{sec2} we define the regional optimal control problem and we state the complete set of assumptions that we consider throughout. We analyze in detail the structures 1-2 and 1-$\H$-2 (the other cases being similar), by providing an explicit construction of the duplicated problem. As a result, we obtain the above-mentioned representation of the value function of the regional optimal control problem and the consequences for its regularity. 

In Section \ref{sec_example} we provide a simple regional optimal control problem, having a structure 1-$\H$-2, modelling for instance the motion of a pedestrian walking in $\Omega_1$ and $\Omega_2$ and having the possibility of taking a tramway along $\H$ at any point of this interface $\H$. 

Section \ref{sec_proofs} gathers the proofs of all results stated in Section \ref{sec2}.

\section{Value function for regional optimal control problems}\label{sec2}
\subsection{Problem and main assumptions} \label{ass}

We assume that: 
\begin{itemize} 
    \item[(H$\H$)] \textit{$\R^N=\Omega_1\cup\Omega_2\cup\H$ with
	  $\Omega_1\cap\Omega_2=\emptyset$ and $\H=\partial
	  \Omega_1=\partial\Omega_2$ being a $C^{1}$-submanifold. \\
	  More precisely, there exists a function $\Psi: \R^N \ds \R$ of class $C^1$ such that 
$\H= \{ x \in \R^N \: | \: \Psi(x)=0\}$} with  $\nabla \Psi \neq 0$ on $\H$. 
\end{itemize}

We consider  the problem of minimizing  the cost of trajectories going from $\xo$ to $\xf$ 
in time $\tf - \too$. These trajectories follows the respective dynamics $f_i, f_\H$ when they are respectively in $\Omega_i, \H$, and pay  
different  costs  
$l_i$, $l_\H$ on $\H,  \Omega_i$ ($i=1,2$).

The tangent bundle of $\H$ is $T\H=\bigcup_{z\in\H}\Big(\{z\}\times T_z\H\Big)$, 
where $T_z\H$ is the tangent space to $\H$ at $z$ (which is
isomorphic to $\R^{N-1}$). For $\phi\in C^1(\H)$ and
$x\in\H$, we denote by $\nabla_\H \phi(x)$ the gradient of $\phi$ at
$x$, which belongs to $T_{x}\H$. 
 The scalar product in $T_z\H$ is denoted by
$\psgH{u}{v}$.  This definition makes sense if both vectors $u,v$  belong to
$T_z\H$ and without ambiguity we will use the same notation when one of the vectors $u,v$ is in $\R^N$. The notation $\psg{u}{v}$  refers  to the usual Euclidean scalar product in $\R^N$. 

We make the following assumptions: 
\begin{itemize} 

 
\item[(Hg)] \textit{Let ${\cal M}$ be a submanifold of $\R^N$ and $A$ a measurable subsets of $\R^m$,  the function $g : {\cal M}\times A \to \R^N$ is a continuous bounded function, $C^1$ and with  Lipschitz continuous derivative with respect to the first variable. 
	More precisely,  there
	exists  $M>0$ such that for any $x \in {\cal M}$
	and $\alpha \in A$,
	$$ |g(x,\alpha) | \leq M . $$ 
	Moreover, there exist $L , L^1>0$ such that for
	any $z, z'\in  {\cal M}$
	 and $\alpha \in A$,
    $$|g(z,\alpha)-g(z',\alpha)|\leq   L \: |z-z'|    , $$
$$
   \Big|\frac{\partial }{ \partial z_j }g(z,\alpha)-\frac{\partial }{ \partial z_j } g(z',\alpha) \Big|\leq   L^1 |z-z'|  \:  , \:  j=1, \dots , N  .
$$
}
\item[(Hf$l_i$)] \textit{Let $A_i$ (i=1,2) be  measurable subsets of $\R^m$. We assume  that $f_i:\Omega_i\times A_i \to \R^N$, $l_i: \Omega_i\times A_i  \to \R$, ($i=1,2$) satisfy Assumption (Hg) for a suitable choice of positive constants $M,L$ and $L^1$. }

\item[(Hf$l_\H$)] \textit{Let $A_\H$  be  measurable subsets of $\R^m$. We assume  that $(x,  f_\H(x,a_\H))  : \H \times A_\H \to T\H$ and $ l_\H : \H \times A_\H \to \R$ 
 satisfy Assumption  (Hg) for a suitable choice of positive constants $M,L$ and $L^1$. }


  \end{itemize}

In this paper we consider optimal trajectories that are decomposed on arcs staying only in $\Omega_1$,  $\Omega_2$ or $\H$ 
and  touch the boundary of  $\Omega_1$, or $\Omega_2$ only at initial or final time.

\paragraph{The problem in the region $\Omega_i$ (for $i=1$ or $2$).}
 The trajectories  $X_i : \R^+ \ds \R^N$  are solutions of
 \begin{equation} \label{stato1}
 \dot{X}_i(t)= f_i (X_i(t), \alpha_i(t))  ,  \quad  \: X_i(t) \in \Omega_i \: \:   \:  \forall t \in  (\too, \tf)
\end{equation} 
 \begin{equation} \label{stato1CIF}
   \: X_i(\too)=\xo \:  , \:  \: X_i(\tf)=\xf \:  \mbox{ with  }  \xo \neq \xf \in \overline{\Omega}_i .
\end{equation} 
The value function $\V_{i}:  \overline{\Omega}_i \times \R^+ \times \overline{\Omega}_i \times \R^+ \ds \R $ is 
\begin{equation*} 
\V_{i}(\xo,\too;\xf,\tf)= \inf \Big\{    \int_{\too}^{\tf} l_i(X_i(t),\alpha_i(t))  \: dt \: : \:   X_i \mbox{ is solution of }  \eqref{stato1}- \eqref{stato1CIF}
\:  \Big\} . 
\end{equation*}
We define  the Hamiltonian 
 $\Hus : \Omega_i \times \R^N \times \R \times  A_i \ds \R$ by
\begin{equation*}
 \Hus(X_i,Q_i,p^0,\alpha_i)=\psg{Q_i}{f_i(X_i,\alpha_i)} + p^0 \: l_i(X_i, \alpha_i)  , \end{equation*}
 and $\Hu : \Omega_i \times \R^N  \times \R \ds \R$ by
\begin{equation*}
  \Hu(X_i,Q_i,p^0)=\sup_{\alpha_i \in A_i} 
 \Hus(X_i,Q_i,p^0,\alpha_i) \:.
\end{equation*}

\paragraph{The problem along the interface $\H$.}
The trajectories  $X_\H :  \R^+ \ds \H$  are solutions of
 \begin{equation} \label{statoh}
 \dot{X}_\H(t)= f_\H (X_\H(t), a_\H(t))  ,  \quad  \: X_\H(t) \in \H \: \:   \:  \forall t \in  (\too, \tf)
\end{equation} 
 \begin{equation} \label{statohCIF}
   \: X_\H(\too)=\xo \:  , \:  \: X_\H(\tf)=\xf \:  \mbox{ with  }  \xo \neq \xf \in \H .
\end{equation} 
The value function $\V_{\H}:  \H \times \R^+ \times \H \times \R^+ \ds \R $ is 
\begin{equation*} 
\V_{\H}(\xo,\too;\xf,\tf)= \inf \Big\{    \int_{\too}^{\tf} l_\H(X_\H(t),a_\H(t))  \: dt \: : \:   X_\H \mbox{ is solution of }  \eqref{statoh}- \eqref{statohCIF}
\:  \Big\} . 
\end{equation*}
We define  the Hamiltonian  $\Hhs :  T\H \times \R \times A_\H  \ds \R$ by
\begin{equation*}
 \Hhs(X_\H,Q_\H,p^0,a_\H)=\psgH{Q_\H}{f_\H(X_\H,a_\H)} + p^0 \: l_\H(X_\H, \a_\H)   \:  , 
  \end{equation*}
and $\Hh : T\H  \times \R \ds \R$ by
\begin{equation*}
 \Hh(X_\H,Q_\H,p^0)=\sup_{a_\H \in A_\H}    \Hhs(X_\H,Q_\H,p^0,a_\H)\:.
 \end{equation*}

\subsection{Analysis of the structure 1-2} \label{sec12}

We describe here  the simplest possible structure: trajectories  consisting of two arcs  living successively in $\Omega_1$, $\Omega_2$ and crossing the interface $\H$ at a given time (see Figure \ref{fig_1}). 
This case has already been studied in the literature. As explained in \cite{ClarkeVinter1}, the jump condition \eqref{conHT12} herefter is a rather straightforward generalization of the problem solved by Snell's Law. Besides, the Pontryagin maximum principle is also well established in this case; we  recall it  hereafter  in detail because it is  interesting to compare this result with the one  obtained for more general structures (see Theorem \ref{mainthm} and Remark  \ref{remmain}).

We make the following \emph{transversal crossing} assumption:

\begin{itemize}
\item[(H 1-2)] There exist a time $t_c \in (\too,\tf)$ and  an optimal trajectory  that starts from $\Omega_1$, stays in $\Omega_1$ in the interval $[\too,t_c)$, does not arrive tangentially  at time $t_c$ on $\H$ and stays in  $\Omega_2$ on the interval $(t_c,\tf]$.
\end{itemize}

Such trajectories are described as follows: 
for each  initial  and final data $(\xo,\too,\xf,\tf)\in  \Omega_1 \times \R^+ \times \Omega_2 \times \R^+_*$,  the trajectory is given by   the vector
$\vsx(t)=(X_1(t),  X_2(t)) : \R^+ \ds  \R^{N}  \times \R^N$
Lipschitz solution of  the system 
\begin{equation} \label{SysHyb12} 
\left\{ \begin{array}{rll}
\dot{X}_1(t) &= f_1 (X_1(t), \alpha_1(t))   & t \in (\too,t_c) \\
\dot{X}_2(t) & =f_2(X_2(t),\alpha_2(t)) & t \in (t_c,\tf)
\end{array}  \ch
\end{equation}
completed with the  mixed   conditions 
\begin{equation} \label{CondMixte12}
X_1(t_0)=\xo , \: \: X_1(t_c) =X_2(t_c),\: \: X_2(\tf)=\xf  , 
\end{equation}
the non tangential conditions 
\begin{equation} \label{Tangcond12} 
 \psg{\nabla \Psi(X_1(t_c^-))}{f_1(X_1(t_c^-), \alpha_1(t_c^-))} \neq 0 \quad \psg{ \nabla \Psi(X_2(t_c^+))}{f_2(X_2(t_c^+),\alpha_2(t_c^+))} \neq 0 ,
\end{equation}
 and the state constraints 
\begin{equation} \label{SCHstay12}
 X_1(t) \in \Omega_1  \:  \: \forall  t \in (\too,t_c)   ,   \qquad  
  X_2(t) \in \Omega_2  \:  \: \forall  t \in (t_c,\tf).
\end{equation}
The cost of such a trajectory is
\begin{equation*} 
\dis \co(\xo,\too;\xf,\tf; \vsx)=\int_{\too}^{t_c} l_1(X_1(t),\alpha_1(t))  \: dt  +\int_{t_c}^{\tf} l_2(X_2(t),\alpha_2(t))   \: dt  .
\end{equation*}
Hence the  value function $\V_{1,2} :  \R^N \times \R^+ \times \R^N \times \R^+  \ds \R$ is given by
\begin{multline*} 
\V_{1, 2}(\xo,\too;\xf,\tf)= \inf \Big\{  \co(\xo,\too;\xf,\tf; \vsx)    \: : \:   \vsx \mbox{ is solution of    \eqref{SysHyb12}-\eqref{CondMixte12}- \eqref{Tangcond12} -\eqref{SCHstay12}, }  \\
 \tf > t_c > \too  \: \:  \Big\}.
\end{multline*}
Under the assumptions (H$\H$),  (Hfl$_i$), (Hfl$_\H$), the results of \cite{GP,HT} apply and 
 for any  $(\xo,\too;\xf,\tf)  \in  \Omega_1 \times \R^+ \times \Omega_2 \times \R^+$ we have 
\begin{equation*} 
\V_{1,2}(\xo,\too;\xf,\tf)= 
\min_{ } \apg  \V_{1}(\xo,\too; x_c, t_c) + \V_{2}(x_c,t_c; \xf,\tf)  \: :  \too < t_c < \tf  , \:  x_c \in \H \chg ,
\end{equation*}
where we recall that $\V_{i}$ is the value function of the problem restricted to the region $\Omega_i$. Moreover, if $\vsx(\cdot)$ is an optimal trajectory for the value  function $\V_{1,2}(\xo,\too;\xf,\tf)$  and $P(\cdot)$ is   the corresponding adjoint vector given by the Pontryagin maximum principle,  then   we have the  continuity condition
\begin{equation*} 
\Hu(X_1(t_c^-),P(t_c^-),p^0)=\Hd(X_2(t_c^+),P(t_c ^+),p^0).
\end{equation*}
and  the  {\it jump condition}  on the adjoint vectors 
\begin{equation} \label{conHT12}
P_2(t_c^+) - P_1(t_c^-)  = \frac{\psg{P_1(t_c^-)}{f_1(t_c^-)-f_2(t_c^+)}+
p^0(l_1(t_c^-)-l_2(t_c^+) )}{\psg{\nabla \Psi(X_2(t_c^+))}{f_2(t_c^+)}} 
\nabla \Psi(X_1(t_c^-))
\end{equation}
where, above, the short notation $f_i(t_c^\pm)$ stands for $f_i(X_i(t_c^\pm),t_c,\alpha_i(t_c^\pm))$ and 
$l_i(t_c^\pm)$ stands for $l_i(x_i(t_c^\pm),t_c,\alpha_i(t_c^\pm))$, ($i=1,2$).

\subsection{Analysis of the structure 1-$\H$-2} \label{sec1H2}
In this section we analyze the  structure with three arcs described in Figure \ref{fig_2}. Precisely, 
given $(\xo,\too,\xf,\tf)\in  \Omega_1 \times \R^+ \times \Omega_2 \times \R^+$  with $\xo \neq \xf$ we make the following assumption:

\begin{itemize}
\item[\hysh] There exist $\too < t_1 < t_2 < \tf$ and  an optimal trajectory  that starts from $\Omega_1$, stays in $\Omega_1$ in the interval $[\too,t_1)$, 
    stays on $\H$ on a   time interval  $[t_1,t_2]$ and  stays in $\Omega_2$ in the interval $(t_2,\tf]$.
\end{itemize}

Such trajectories are described as follows: 
for each  initial  and final data $(\xo,\too,\xf,\tf)\in  \Omega_1 \times \R^+ \times \Omega_2 \times \R_*^+$,  the trajectory will be given by   the vector
$\vsx(t)=(X_1(t), X_\H(t), X_2(t)) : \R^+ \ds  \R^{N} \times \H \times \R^N$
Lipschitz solution of  the  system 
\begin{equation} \label{SysHyb} 
\left\{ \begin{array}{rll}
\dot{X}_1(t) &= f_1 (X_1(t), \alpha_1(t))   & t \in (\too,t_1) \\
\dot{X}_\H(t) &= f_\H (X_\H(t), \a_\H(t)) & t \in (t_1,t_2) \\
\dot{X}_2(t) & =f_2(X_2(t),\alpha_2(t)) & t \in (t_2,\tf)
\end{array}  \ch
\end{equation}
with mixed conditions
\begin{equation} \label{CondMixte}
X_1(t_0)=\xo , \: \: X_1(t_1) \in \H ,\: \:X_1(t_1)=X_\H(t_1) ,\: \:X_\H(t_2) \in \H , \: \:X_2(t_2)=X_\H(t_2) ,\: \: X_2(\tf)=\xf  \: 
\end{equation}
and the state constraints 
\begin{equation} \label{SCHstay}
 X_1(t) \in \Omega_1  \:  \: \forall  t \in (\too,t_1)   , \quad  X_\H \in \H \:  \: \forall t \in (t_1,t_2) ,  \quad    
  X_2(t) \in \Omega_2  \:  \: \forall  t \in (t_2,\tf).
\end{equation}
The cost of such a trajectory is
\begin{multline*}
\dis \co(\xo,\too;\xf,\tf; \vsx)=\int_{\too}^{t_1} l_1(X_1(t),\alpha_1(t))  \: dt  + \int_{t_1}^{t_2} l_\H(X_\H(t),\a_\H(t))  \:  dt \\
+\int_{t_2}^{\tf} l_2(X_2(t),\alpha_2(t))   \: dt  .
\end{multline*}
Our aim is  to characterize the  value function $\V_{1,\H,2} :  \Omega_1 \times \R^+ \times \Omega_2 \times \R^+  \ds \R$
\begin{multline} \label{valoreUHD}
\V_{1,\H,2}(\xo,\too;\xf,\tf)= \inf \Big\{  \co(\xo,\too;\xf,\tf; \vsx)    \: : \:   \vsx \mbox{ is solution of \eqref{SysHyb}-\eqref{CondMixte}-\eqref{SCHstay}, } \\
 \tf > t_2 > t_1 > \too  \: 
\:  \Big\}.
\end{multline}

\begin{remark} 
This definition does not include the cases where $\xo \in \H$ and/or $\xf \in \H$. However, it can be  modified in order to involve only  vectors $X_1$, $X_2$ or $X_\H$. Moreover, note that  if both $\xo, \xf  \in \H$  then  $\V_{1,\H,2}(\xo,\too;\xf,\tf)=\V_\H(\xo,\too;\xf,\tf)$. 
\end{remark}

Herafter, we use the following notations.

\paragraph{Notations.}
Let   $u=u(\xo,\too;\xf,\tf):  \big(\Omega_1 \times \R^+  \times \Omega_2 \times \R_*^+ \big)    \rightarrow \R$ be a generic function. 
We   denote by $\nabla _{\xo} u$, $\nabla _{\xf} u$ the gradients with respect to the first and the second state variable respectively,  so $\nabla _{\xo} u$ and $\nabla _{\xf} u$ take values in $\R^N$. 
We denote by $u_{\too}$ and $u_{\tf}$ the partial derivatives with respect to the first and the second time variable   respectively, so  $u_{\too}$ and $u_{\tf}$ take values in $\R$.

 If $(\xo,\too;\xf,\tf) \in \Omega_1 \times \R^+ \times \H \times \R^+_*$ we  define $\nabla^\H_{\xf} u$  such that 
$(\xf, \nabla^\H_{\xf} u) \in T\H$. 

If $(\xo,\too;\xf,\tf) \in \H \times \R^+ \times \Omega_2 \times \R^+_*$  we  define $\nabla^\H_{\xo} u$  such that 
$(\xo, \nabla^\H_{\xo} u) \in T\H$.

\paragraph{Definition of the duplicated problem.}
The main ingredient of our analysis is the construction of the duplicated problem (following  \cite{Dmitruk}),
the advantage being that the latter  will be a  classical (nonregional)  problem in  higher dimension.  The idea is to change the time variable to let the possible optimal trajectories evolve ``at the same time" on the three arcs: the one on $\Omega_1$, the one on $\H$ and the one on $\Omega_2$. 
In this duplicated  optimal control problem we will not need to  impose  the mixed conditions  
\eqref{CondMixte}  and the state constraints  \eqref{SCHstay}. 
Therefore  we will be able to characterize the  value function by an Hamilton-Jacobi equation, apply the usual Pontryagin maximum principle and exploit the classical link (sensitivity relations) between them. 

We set $V=(A_1 \times [0,T]) \times  (A_\H \times [0,T]) \times (A_2 \times [0,T])$, for $T>0$ large enough. For fixed $\To,\Tu \in \R^+$ the {\it admissible  controls} are $ \vcr(\tau)=(\vu(\tau),\wu(\tau),\vh(\tau),\wh(\tau),\vd(\tau),\wdu(\tau)) \in  L^\infty([\To,\Tu] ; V)$. \\
The {\it admissible trajectories} are Lipschitz continuous vector functions 
$$
 \vsz(\tau)=(Y_1(\tau),\rou(\tau),Y_\H(\tau),\roh(\tau), Y_2(\tau),\rod(\tau))  : (\To, \Tu) \ds \Omega_1 \times \R_*^+ \times \H \times \R_*^+ \times \Omega_2 \times \R_*^+
 $$
 solutions of  the so-called \emph{duplicated system}
 \begin{equation} \label{SysHybRip} 
\left\{ \begin{array}{rll}
{Y}^\prime_1(\tau) &= f_1 (Y_1(\tau), \vu(\tau)) \wu(\tau)   & \tau \in (\To,\Tu) \\
\rou^\prime(\tau) & = \wu(\tau) & \tau \in (\To,\Tu)  \\
{Y}^\prime_\H(\tau) &= f_\H (Y_\H(\tau), \vh(\tau)) \wh(\tau)   & \tau \in  (\To,\Tu)  \\
\roh^\prime(\tau) & = \wh(\tau) & \tau \in  (\To,\Tu) \\
{Y}^\prime_2(\tau) &= f_2 (Y_2(\tau), \vd(\tau)) \wdu(\tau)   & \tau \in  (\To,\Tu)  \\
\rod^\prime(\tau) & = \wdu(\tau) & \tau \in (\To,\Tu)  \\
\end{array}  \ch
\end{equation}
with initial and final conditions 
\begin{equation} \label{CondIniFin} 
\vsz(\To)=\Zo  \:  \mbox{ , } \:  \vsz(\Tu)=\Zu   .
\end{equation}
Note that to take into account the mixed conditions on the original problem, we will allow
 initial and final  state $\Zo$, $\Zu$  in $\overline{\Omega}_1 \times \R_*^+ \times \H \times \R_*^+ \times \overline{\Omega}_2 \times \R_*^+$. \\
 More precisely, given  $(\Zo,\To)$, $(\Zu,\Tu)$ $\in (\overline{\Omega}_1 \times \R_*^+ \times \H \times \R_*^+ \times \overline{\Omega}_2 \times \R_*^+) \times \R^+$ we  consider the subset of admissible trajectories
$$
{\cal  Z}_{(\Zo,\To),(\Zu,\Tu)}= \Big\{ \vsz \in Lip((\To, \Tu); \Omega_1 \times \R_*^+ \times \H \times \R_*^+ \times \Omega_2 \times \R_*^+) \: : \:  \mbox{  there exists an admissible  control }
$$
$$
\quad \quad \quad \quad\quad \quad     \vcr \in L^\infty([\To,\Tu] ; V)  \mbox{ such that } 
\vsz \mbox{ is a solution of \eqref{SysHybRip}-\eqref{CondIniFin}}  \Big \} . 
$$
For each admissible trajectory $\vsz$ we  consider the cost functional
\begin{equation*} 
\dis \cor(\vsz)=\int_{\To}^{\Tu} \Big(  l_1(Y_1(\tau),\vu(\tau))\wu(\tau)   +l_\H(Y_\H(\tau),\vh(\tau))\wh(\tau)  
+ l_2(Y_2(\tau),\vd(\tau))\wdu(\tau) \Big) \: d\tau \: 
 \end{equation*}
and hence the  value function  
  $\Vr :  \big(\overline{\Omega}_1 \times \R_*^+ \times \H \times \R_*^+ \times \overline{\Omega}_2 \times \R_*^+   \times \R^+ \big)^2 \ds \R $ is  defined  by  
\begin{equation} \label{valore}
\Vr(\Zo,\To;\Zu,\Tu)= \inf \Big\{  \cor(\vsz)    \: : \:   \vsz  \in    {\cal  Z}_{(\Zo,\To),(\Zu,\Tu)}   \Big\} .
\end{equation}

\paragraph{Link between the regional optimal control problem and the duplicated problem.}
To establish  the   link between the original and the  duplicated problem,  given   $(\xo,\too;\xf,\tf) \in \Omega_1 \times \R^+_* \times \Omega_2 \times \R^+_*$, we define the  submanifold of $\R^{6(N+1)}$
\begin{multline*}
\Sm(\xo,\too;\xf,\tf) = \Big\{  (\Zo,\Zu) \in  \R^{6(N+1)} \: :  \: \Zo=(\xo,\too,x_1,t_1,x_2,t_2) ,   \\
 \Zu=(x_1,t_1,x_2,t_2,\xf,\tf)  \,  
\mbox{ with } \:  x_1 \in \H ,  x_2 \in \H, \: \mbox{ and }  \tf > t_2 >  t_1 > \too \Big\}.
\end{multline*}
The following result says that  the original value function is  the minimum of the value functions  $ \Vr(\Zo,\To;\Zu,\Tu)$ 
restricted to the submanifold  $\Sm(\xo,\too;\xf,\tf)$.

\begin{proposition} \label{infvalues}
Under the assumptions (H$\H$),  (Hfl$_i$) and (Hfl$_\H$), given   $(\xo,\too;\xf,\tf) \in \Omega_1 \times \R^+ \times \Omega_2 \times \R^+$,  we have 
\begin{equation} \label{tesilem}
\V_{1,\H,2}(\xo,\too;\xf,\tf)= \min  \Big\{  \Vr(\Zo,\To;\Zu,\Tu)    \: : \:   (\Zo,\Zu) \in \Sm(\xo,\too;\xf,\tf), \ 0 \leq \To < \Tu \Big\} .
\end{equation}
\end{proposition} 

Proposition \ref{infvalues} is proved in Section \ref{proof_infvalues}.


\paragraph{Application of the usual Pontryagin maximum principle  to the duplicated problem.}
Let us introduce several further notations.

In order to write the partial derivatives of $\Vr$ at points $\vsz  \in \Omega_1 \times \R_*^+ \times \H \times \R_*^+ \times \Omega_2 \times \R_*^+$
  we enumerate the  space variables  as follows: $(\vsz_0,\vsz_1)=\big( (1,2,3,4,5,6) , (7,8,9,10,11,12) \big) $ therefore $\pt_i \Vr $ takes values in 
$\R$ for $i=2,4,6,8,10,12$;  $\ps_i \Vr $ takes values in $\R^N$ for $i=1,5,7,11$ and  $ \nabla_{\H,i} \Vr $ in $T \H$ for $i=3,9$.  
We set 
\begin{equation*} 
\nabla_{\vsz_0}  = ( \ps_1 \Vr ,\pt_2 \Vr,\ps_{\H,3} \Vr ,\pt_4 \Vr ,\ps_5 \Vr ,\pt_6 \Vr)  , \quad
\Vr_{t_0}(\vsz_0,t_0,\vsz_1, t_1) = - \frac{\partial }{ \partial t_0 } \Vr(\vsz_0,t_0,\vsz_1, t_1),
\end{equation*}
\begin{equation*}
 \nabla_{\vsz_1}  \Vr= ( \ps_7 \Vr ,\pt_8 \Vr,\ps_{\H,9} \Vr ,\pt_{10} \Vr ,\ps_{11} \Vr ,\pt_{12} \Vr) ,  \quad 
 \Vr_{t_1}(\vsz_0,t_0,\vsz_1, t_1) = - \frac{\partial }{ \partial t_1 } \Vr(\vsz_0,t_0,\vsz_1, t_1).
\end{equation*}
Moreover, we  respectively denote by $D^{+}_{\vsz_0} \Vr $ and  $D^{-}_{\vsz_0} \Vr $ (or $D^{+}_{\vsz_1} \Vr $ and  $D^{-}_{\vsz_1} \Vr $) the classical super- and sub-differential  in  the space variables $1,3,5,7,9$. 

Given $\vcr=(\vu,\wu,\vh,\wh,\vd,\wdu) \in V$,  $\vsz=(Y_1,\rho_1,Y_\H,\rho_\H,Y_2,\rho_2) \in 
\R^N \times \R \times T\H \times  \R \times \R^N  \times \R $   
and  $\mathbb{Q}=(Q_1,Q_2,Q_3,Q_4,Q_5,Q_6) \in \R^N \times \R \times T_{Y_\H}\H \times  \R \times \R^N  \times \R$, 
we define the Hamiltonian
\begin{multline*}
\Hrs(\vsz,\mathbb{Q},p^0,\vcr)= \psg{Q_1}{f_1(Y_1,\vu)\wu} + Q_2  \: \wu +  p^0  \: l_1(Y_1,\vu)\wu + \psgH{Q_3}{f_\H(Y_\H,\vh)\wh} \\
+ Q_4  \: \wh +  p^0  \: l_\H(Y_\H,\vh)\wh 
+\psg{Q_5}{f_2(Y_2,\vd)\wdu} + Q_6  \: \wdu +  p^0  \: l_2(Y_2,\vd)\wdu  
\end{multline*}
and we set
$$
\Hr(\vsz,\mathbb{Q},p^0)= \sup_{ \vcr \in V}  \Hrs(\vsz,\mathbb{Q},p^0,\vcr) .
$$

The application of the usual Pontryagin maximum principle  to the duplicated optimal control problem leads to the following lemma.

 \begin{lemma}  \label{PMPripa}
Under the assumptions (H$\H$),  (Hfl$_i$) and  (Hfl$_\H$), let $(\Zo,\To)$, $(\Zu,\Tu)$ $\in  (\Omega_1 \times \R_*^+ \times \H \times \R_*^+ \times \Omega_2 \times \R_*^+ )\times \R^+$
and let $\vsz(\cdot) \in {\cal  Z}_{(\Zo,\To),(\Zu,\Tu)}$ be an optimal trajectory for the value function $\Vr(\Zo,\To;\Zu,\Tu)$
 defined in \eqref{valore}. Assume that  $\vcr(\cdot)$ is the corresponding optimal  control.  \\
 There exist $p^0 \leq 0$ and a piecewise absolutely continuous mapping 
$$
\dis \vapz(\cdot) =(\vayu(\cdot), \varu(\cdot), \vayh(\cdot), \varh(\cdot), \vayd(\cdot), \vard(\cdot)) : \R^+ \ds \R^N \times \R \times T_{Y_\H}\H \times  \R \times \R^N  \times \R
$$ 
(adjoint vector) with $(\vapz(\cdot),p^0) \neq (0,0)$,
such that the {\it extremal }  lift $(\vsz(\cdot),\vapz(\cdot),p^0,\vcr(\cdot))$ is solution of 
 \begin{equation*} 
\vsz^\prime(\tau)= \frac{\partial \Hrs}{\partial P} \: (\vsz(\tau),\vapz(\tau),p^0,\vcr(\tau))    , \quad  
\vapz^\prime(\tau)= -\frac{\partial \Hrs}{\partial Z} \: (\vsz(\tau),\vapz(\tau),p^0,\vcr(\tau))  
\end{equation*}  
for almost every $\tau \in (T_0,T_1)$.
Moreover,  the  maximization condition 
\begin{equation} \label{condmax}
\Hrs(\vsz(\tau),\vapz(\tau),p^0,\vcr(\tau))=\max_{\vcr \in V} \:   \Hrs(\vsz(\tau),\vapz(\tau),p^0,\vcr) \: \:( = \Hr(\vsz(\tau),\vapz(\tau),p^0 ) \: )
\end{equation}
holds for almost every $\tau \in (T_0,T_1)$.  \\
If  $\Zo=(\xo,\too,x_1,t_1,x_2,t_2)$, $\Zu=(x_1,t_1,x_2,t_2,\xf,\tf)$ $\in\Sm(\xo,\too;\xf,\tf)$ then the following 
 transversality condition holds:  there  exist $\nu_1, \nu_2 \in \R$ such that 
\begin{eqnarray}  
\label{Jumpt1}
\varh(T_0)&=&\varu(T_1) \\
\label{Jumpt2}
\vard(T_0)&=& \varh(T_1) \\
\label{Jumps1} 
\vayh(T_0)&=&\vayu(T_1)+\nu_1 \:  \nabla \Psi(x_1) \\
\label{Jumps2} 
\vayd(T_0)&=& \vayh(T_1)+ \nu_2  \:  \nabla \Psi(x_2) . 
\end{eqnarray}
\end{lemma}

We provide a proof of Lemma \ref{PMPripa} in Section \ref{proof_PMPripa}.

\paragraph{Sensitivity relations.}
In order to establish  the  link  between the adjoint vector and the gradient of  the value function $\Vr$, we assume the uniqueness of the extremal lift: 
\begin{itemize}
\item[(Hu)]  We assume that the optimal trajectory $\vsz(\cdot)$  in Lemma \ref{PMPripa} admits a unique   extremal  lift $(\vsz(\cdot),\vapz(\cdot),p^0,\vcr(\cdot))$ which is moreover normal, i.e., $p^0=-1$. 
\end{itemize} 

The assumption of uniqueness of the solution of the optimal control problem and of uniqueness of its extremal lift (which is then moreover normal) is closely related to the differentiability properties of the value function. We refer to \cite{AubinFrankowska,ClarkeVinter} for precise results on differentiability properties of the value function and to \cite{CannarsaSinestrari,RT1,RT2,Stefani} for results on the size of the set where the value function is differentiable. For instance for control-affine systems the singular set of the value function has Hausdorff $(N-1)$-measure zero, whenever there is no optimal singular trajectory (see \cite{RT2}), and is a stratified submanifold of $\R^N$ of positive codimension in an analytic context (see \cite{T2006}). These results essentially say that, if the dynamics and cost function are $C^1$, then the value function is of class $C^1$ at ``generic" points.
Moreover, note that the property of having a unique extremal lift, that is moreover normal, is generic in the sense of the Whitney topology for control-affine systems (see \cite{Chitour_Jean_TrelatJDG,Chitour_Jean_Trelat} for precise statements).

We have the following result.

\begin{proposition} \label{TEOLINKBASE} \label{proof_TEOLINKBASE}
Assume (H$\H$), (Hfl$_i$) and  (Hfl$_\H$). 
 Let $(\Zo,\To)$, $(\Zu,\Tu) \in \big(\overline{\Omega}_1 \times \R_*^+ \times \H \times \R_*^+ \times \overline{\Omega}_2 \times \R_*^+ \big)  \times \R^+$  and let $\vsz(\cdot) \in {\cal  Z}_{(\Zo,\To),(\Zu,\Tu)}$ be an optimal trajectory for the value function $\Vr(\Zo,\To;\Zu,\Tu)$
 defined in \eqref{valore}. Let  $\vapz$ be  the corresponding  absolutely continuous adjoint vector given by Theorem  \ref{PMPripa}.  Then:
 \begin{itemize}
 \item[(i)] For any time $\tau$ in the closed interval   $[\To,\Tu]$ we have
 \begin{equation}  \label{linkbase0}
 D^{-}_{\Zo}  \Vr(\vsz(\tau),\tau;\Zu,\Tu) \subseteq - \vapz (\tau)\subseteq D^+_{\Zo}  \Vr (\vsz (\tau),\tau; \Zu,\Tu) \quad \quad
  \end{equation} 
in the sense that either   $D^{-}_{\Zo}  \Vr(\vsz(\tau),\tau;\Zu,\Tu)$ is empty or the function $\tau \mapsto  \Vr(\vsz(\tau),\tau;\Zu,\Tu) $   
 is differentiable and then $ D^{-}_{\Zo}   \Vr = D^+_{\Zo}  \Vr $ at  this point. \\
Moreover, when assumption (Hu) holds  the function $\tau \mapsto   \Vr(\vsz(\tau),\tau;\Zu,\Tu) $ 
is differentiable for every time in $[\To,\Tu]$, thus
 \begin{equation} \label{linkbaseT0}
 \nabla_{\Zo}  \Vr(\vsz(\tau),\tau;\Zu,\Tu) = -\vapz (\tau) \quad   \: \quad  \forall \tau \in   [\To,\Tu]  . 
 \end{equation} 
\item[(ii)] For any time $\tau$ in the closed interval   $[\To,\Tu]$ we have
 \begin{equation}  \label{linkbase1}
 D^{-}_{\Zu}  \Vr(\Zo,\To;\vsz(\tau),\tau) \subseteq  \vapz (\tau)\subseteq D^+_{\Zu}  \Vr (\Zo,\To;\vsz(\tau),\tau) \quad \quad
  \end{equation} 
in the sense that either   $D^{-}_{\Zu}  \Vr(\Zo,\To;\vsz(\tau),\tau)$ is empty or the function $\tau \mapsto   \Vr(\Zo,\To;\vsz(\tau),\tau)$   
 is differentiable and then $ D^{-}_{\Zu}   \Vr = D^+_{\Zu}  \Vr $ at  this point. \\
Moreover, when assumption (Hu) holds  the function $\tau \mapsto   \Vr(\Zo,\To;\vsz(\tau),\tau)$ 
is differentiable for every time in  $[\To,\Tu]$, thus
 \begin{equation} \label{linkbaseT1}
 \nabla_{\Zu}  \Vr(\Zo,\To;\vsz(\tau),\tau) = \vapz (\tau) \quad   \: \quad  \forall \tau \in   [\To,\Tu]  . 
 \end{equation} 
\end{itemize} 
\end{proposition} 

Proposition \ref{TEOLINKBASE} is proved in Section \ref{proof_TEOLINKBASE}.

\begin{remark} \label{remHJtotu}
It is useful  to write  equalities  \eqref{linkbaseT0} and \eqref{linkbaseT1} as a single equality.  
We have indeed 
 \begin{equation} \label{linkbaseT01}
- \nabla_{\Zo}  \Vr(\vsz(\tau),\tau;\Zu,\Tu) = \vapz (\tau)=  \nabla_{\Zu}  \Vr(\Zo,\To;\vsz(\tau),\tau) \quad   \: \quad  \forall \tau \in  [\To,\Tu]  \: 
 \end{equation} 
that is, more precisely,
\begin{eqnarray} \label{linkbase1-7}
- \nabla_1 \Vr(\vsz(\tau),\tau;\Zu,\Tu)= \vayu (\tau)=  \nabla_7 \Vr(\Zo,\To;\vsz(\tau),\tau) \quad   \: \quad  \forall \tau \in    [\To,\Tu]  \: \\
\label{linkbase2-8}
-\partial_2  \Vr(\vsz(\tau),\tau;\Zu,\Tu)= \varu(\tau)=\partial_8 \Vr(\Zo,\To;\vsz(\tau),\tau) \quad   \: \quad  \forall \tau \in    [\To,\Tu]  \: \\
\label{linkbase3-9}
- \nabla^\H_3 \Vr(\vsz(\tau),\tau;\Zu,\Tu)= \vayh (\tau)=  \nabla^\H_9 \Vr(\Zo,\To;\vsz(\tau),\tau) \quad   \: \quad  \forall \tau \in   [\To,\Tu]  \: \\
\label{linkbase4-10}
-\partial_4  \Vr(\vsz(\tau),\tau;\Zu,\Tu)= \varh(\tau)=\partial_{10} \Vr(\Zo,\To;\vsz(\tau),\tau) \quad   \: \quad  \forall \tau \in    [\To,\Tu]  \: \\
\label{linkbase5-11}
- \nabla_5 \Vr(\vsz(\tau),\tau;\Zu,\Tu)= \vayd (\tau)=  \nabla_{11} \Vr(\Zo,\To;\vsz(\tau),\tau) \quad   \: \quad  \forall \tau \in    [\To,\Tu]  \: \\
\label{linkbase6-12}
-\partial_6  \Vr(\vsz(\tau),\tau;\Zu,\Tu)= \vard(\tau)=\partial_{12} \Vr(\Zo,\To;\vsz(\tau),\tau) \quad   \: \quad  \forall \tau \in   [\To,\Tu]  \:. 
\end{eqnarray} 
Note that at times $\To$ and $\Tu$ the gradients are naturally   defined as the limits of the gradients in the open interval  $(\To,\Tu)$.
\end{remark}

\paragraph{Application to the regional optimal control problem: main result.}
We  now  establish a result that is analogous  to the one obtained for the structure 1-2. We first remark that for this structure one cannot directly define a global  adjoint vector, therefore  its role will be played by the limit of the gradient of the value function (vectors $Q_1$, $Q_2$, $Q_\H$ below).
The main result is the following.

\begin{theorem} \label{mainthm} 
Under the assumptions (H$\H$), (Hfl$_i$), (Hfl$_\H$)  and (Hu), for any  $(\xo,\too;\xf,\tf)  \in  \Omega_1 \times \R^+ \times \Omega_2 \times \R_*^+$ we have 
\begin{multline*} 
\V_{1,\H,2}(\xo,\too;\xf,\tf)= \min \big\{  \V_{1}(\xo,\too;,x_1, t_1)+\V_{\H}(x_1, t_1; x_2, t_2) + \V_{2}(x_2,t_2; \xf,\tf)  \ :\ \\
t_0 < t_1 < t_2  < \tf\: \: x_1 , x_2 \in \H \big\} .
\end{multline*}
 Let $\vsx(\cdot)$ be an  optimal trajectory for the value  function $\V_{1,\H,2}(\xo,\too;\xf,\tf)$  defined by \eqref{valoreUHD} and let
 $$
 Q_1( t_1^-)= - \lim_{t \ds t_1^- }  \nabla_{\xo} \: \V_{1,\H,2}  \big( X_1(t), t ;\xf,\tf \big) .
$$

$$
Q_2(t_2^+)=\lim_{t \ds t_2^+} \nabla_{\xf} \: \V_{1,\H,2}  \big(  \xo,\too;X_2(t), t  \big) .
$$
$$
Q_\H(t_1^+)= - \lim_{t \ds t_1^+} \nabla^\H_{\xo} \:  \V_{1,\H,2}  \big( X_\H(t), t ;\xf,\tf  \big) 
$$
 $$
Q_\H(t_2^-)= \lim_{t \ds t_2^-} \nabla^\H_{\xf} \: \V_{1,\H,2}  \big(  \xo,\too; X_\H(t), t   \big) .
$$
We have  the  continuity conditions 
\begin{equation} \label{ContTu}
\Hu(X_1(t_1^-),Q_1(t_1^-),p^0)=\Hh(X_\H(t_1^+),Q_\H(t_1^+),p^0)
\end{equation}
\begin{equation} \label{ContTd}
\Hh(X_\H(t_2^-), Q_\H(t_2^-),p^0)=\Hd(X_2(t_2^+),Q_2(t_2^+),p^0).
\end{equation}
Moreover, there exist $\nu_1 , \nu_2 \in \R$ such that 
\begin{equation} \label{saltoH}
\begin{split}
Q_\H(t_1^+)&=Q_1(t_1^-)+ \nu_1 \: \nabla \Psi(X_1(t_1^-)) ,\\
Q_2(t_2^+)&=Q_\H(t_2^-)+ \nu_2 \: \nabla \Psi(X_2(t_2^+)) . 
\end{split}
\end{equation}
Moreover, if $\psg{\nabla \Psi(X_1(t_1^-)}{f_1(t_1^-)} \neq 0$ and $\psg{ \nabla \Psi(X_2(t_2^+))}{f_2(t_2^+)} \neq 0$
then 
\begin{equation} \label{tesinu1}
 \nu_1 = \frac{\psg{Q_\H(t_1^+)}{f_1(t_1^-)} -\psgH{Q_1(t_1^-)}{f_\H(t_1^+)}+ p^0 \: ( l_1(t_1^-)- l_\H(t_1^+) ) }{\psg{\nabla \Psi(X_1(t_1^-))}{f_1(t_1^-)}}
\end{equation} 
and 
\begin{equation} \label{tesimu2}
\nu_2 = \frac{\psgH{Q_\H(t_2^-)}{f_\H(t_2^-)}- \psg{Q_\H(t_2^-)}{f_2(t_2^+)}+ p^0 \: ( l_\H(t_2^-)-l_2(t_2^+))}{\psg{ \nabla \Psi(X_2(t_2^+))}{f_2(t_2^+)}}
\end{equation}
where we used the  short notations $f_i(t_i^\pm) = f_i(X_i(t_i^\pm),\alpha_i(t_i^\pm))$ and 
$l_i(t_i^\pm) = l_i(X_i(t_i^\pm),\alpha_i(t_i^\pm))$,  with $i\in\{1,2,\H\}$.
\end{theorem}

Theorem \ref{mainthm} is proved in Section \ref{proof_mainthm}.

\begin{remark} \label{remmain} 
 Note the similarity between the jump conditions  \eqref{saltoH}-\eqref{tesimu2} and the  ones in the transversal case \eqref{conHT12}: the difference is  due to the  fact that $\H$  is of codimension $1$.   
 \end{remark}

\subsection{More general structures}  
Proceeding as in Section \ref{sec1H2}, the analogue of Proposition \ref{infvalues} is obtained for any other structure 1-2-$\H$-1, 1-$\H$-1-2, 1-2-$\H$-2, etc, in a similar way. For each given such structure, the duplication technique permits to lift the corresponding regional control problem to a classical (i.e., non-regional) optimal control problem in higher dimension, and then the value function of the regional optimal control problem is written as the minimum of the value function of the high-dimensional classical optimal control problem over a submanifold, this submanifold representing the junction conditions of the regional problem (continuity conditions on the state and jump conditions on the adjoint vector).

For example, consider optimal trajectories with the structure 2-$\H$-2-1, i.e., trajectories starting in $\Omega_2$, staying in $\Omega_2$ along the time interval $[\too,t_1)$, then lying in $\H$ on $[t_1,t_2]$, then going back to $\Omega_2$ on $(t_2,t_3]$ and finally staying in $\Omega_1$ in the time interval $(t_3,\tf]$.   
Then, the duplicated problem has four arcs and is settled in dimension $8$. The whole approach developed previously can be applied as well and we obtain the corresponding analogues of Proposition \ref{infvalues} and then of Theorem \ref{mainthm}. 

In such a way, all possible structures can be described as composed of a finite succession of arcs, and are analyzed thanks to the duplication technique. If the structure has $N$ arcs then  the duplicated problem is settled in dimension $2N$. 

As already said, from a practical point of view it is reasonable to restrict the search of optimal trajectories over all possible trajectories having only a finite number of switchings. This is always what is done in practice because, numerically and in real-life implementation, the Zeno phenomenon is not desirable.
Under such an assumption, our approach developed above shows that the value function of the regional optimal control problem can be written as
\begin{equation*}  \label{INFGEN}
U = \inf \{ \V_{1,2},\ \V_{1,\H,2} ,\  \V_{1,2,\H,2}  , \ldots \} ,
\end{equation*} 
where each of the value functions $\V_\star$ is itself the minimum of the value function of a classical optimal control problem (in dimension that is the double of the number of switchings of the corresponding structure) over terminal points running in some submanifold.
An interesting consequence is that:
\begin{quote}
\emph{The regularity of the value function $U$ of the regional optimal control problem is the same (i.e., not more degenerate) than the one of the higher-dimensional classical optimal control problem that lifts the problem}. 
\end{quote}

The lifting duplication technique may thus be seen as a kind of \emph{desingularization}, showing that the value function of the regional optimal control problem is the minimum over all possible structures of value functions associated with classical optimal control problems settled over fixed structures, each of them being the restriction to some submanifold of the value function of a classical optimal control problem in higher dimension.

In particular, if for instance all value functions above are Lipschitz then the value function of the regional optimal control problem is Lipschitz as well.
Note that Lipschitz regularity is ensured if there is no abnormal minimizer (see \cite{TJOTA}), and this sufficient condition is generic in some sense (see \cite{Chitour_Jean_TrelatJDG,Chitour_Jean_Trelat}).

Such a regularity result is new in the context of regional  optimal control problems. 

\begin{remark}
In this paper, for the sake of simplicity we have analyzed regional problems in $\R^N$. Since all  arguments are local, the same procedure can be applied to regional problems settled on a smooth manifold, which is stratified as $M=M^0 \cup M^1\cup \dots M^N$ (disjoint union) where $M^j$ is a $j$-dimensional embedded submanifold of $M$.     
\end{remark}

\begin{remark}
Our results can also be straightforwardly extended to time-dependent dynamics and running costs, and to regions $\Omega_i(t)$ depending on time, always assuming at least a $C^1$-dependence. 
\end{remark}

\subsection{What happens in case of Zeno phenomenon?}
In case the \emph{Zeno phenomenon} occurs, optimal trajectories oscillate for instance between two regions $\Omega_1$ and $\Omega_2$ an infinite number of times over a compact time interval. 

If the number of switchings is countably infinite, then the above procedure can, at least formally, be carried out, but then the duplicated (lifted) problem is settled in infinite (countable) dimension. In order to settle it rigorously, much more functional analysis work would be required. Anyway, formally the value function is then written as an infimum of countably many value functions of classical optimal control problems, but even if the latter are regular enough (for instance, Lipschitz), taking the infimum may break this regularity and create some degeneracy.

If the number of switchings is uncountably infinite, the situation may even go worst. The duplication technique cannot be performed, at least in the form we have done it, and we do not know if there would exist a somewhat related approach to capture any information. The situation is widely open there.
We are not aware of any example of a regional (or, more generally, hybrid) optimal control problem for which the set of switching points of the optimal trajectory would have a fractal structure. Notice the related result stated in \cite{Ag}, according to which, for smooth bracket generating single-input control-affine systems with bounded scalar controls, the set of switching points of the optimal bang-bang controls cannot be a Cantor set.

\section{Example} \label{sec_example}
As an example  we consider here a  simple regional optimal control problem where it is easy to see that a trajectory of the form 1-$\H$-2 is the best possible choice. The idea is to model situations where it is optimal  to move along the interface $\H$ as long as possible. One can think, for example, of a pedestrian walking in $\Omega_1$ and $\Omega_2$ with the possibility of taking a tramway along $\H$ at any point of this interface $\H$. 

More generally, this example models any  problem where moving along a direction is much faster and/or cheaper than along others.  

In $\R^2$ we set $\Omega_1=\{ (x,y) \: : \: y < 0\}$, $\Omega_2=\{ (x,y) \: : \: y > 0\}$ and $\H=\{ (x,y) \: : \: y =0\}$. \\
We choose the dynamics 
 \begin{equation*}
f_1 (X_1, \alpha_1)=\apt \begin{array}{l}
\cos(\alpha_1) \\
\sin(\alpha_1) 
 \end{array}  \cht
 , \:  \quad 
 f_\H(X_\H,\alpha_\H)=10
 , \:   \quad
  f_2 (X_2, \alpha_2)=\apt \begin{array}{l}
\cos(\alpha_2) \\
\sin(\alpha_2) 
 \end{array}  \cht
\end{equation*}
where the controls $\alpha_i$ take values on $[-\pi,\pi]$. We consider the minimal time problem, 
therefore our aim is to compute the  value function 
$$
U(\xo,0;\xf)= \inf \apg   \tf  \:   :   \:  \dot{\vsx}(t) = f(\vsx(t),\a(t)) \quad \mbox{ with }\quad \:  \vsx(t^0)=x^0,  \quad \vsx(t^f)=x^f \chg ,  
$$ 
where  the dynamics $f$ coincide with $f_1$,$f_2$, $f_\H$ respectively in $\Omega_1$, $\Omega_2$, $\H$.

We  analyze the case where we start  from a point $(x_0,y_0)$ in $\Omega_1$ and we aim to reach a point $(x_1,y_1)$ in $\Omega_2$ with 
$x_1 > x_0$. In $\Omega_1$, the dynamics $f_1$ allow  to move with constant velocity equal to one in any direction, therefore  it is clear that the best choice is to go ``towards $\H$ but also in the direction of $x_1$".  Indeed, if we compare  on Figure  \ref{figexample} below  the dotted trajectory and the black one, they spend the same  time in $\Omega_1$, but on $\H$ the dotted one is not the minimal time. 
Therefore the black one is  a better choice.

 \begin{figure}[h]
\begin{center}
{ \resizebox{11cm}{!}{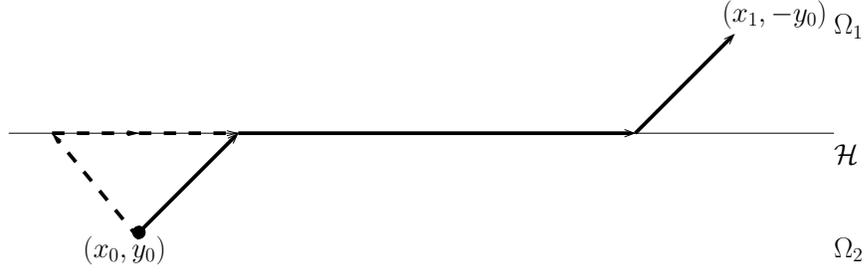} }
\end{center}
\caption{Going ``to the left" is not optimal.}\label{figexample}
\end{figure}

For this reason, and since the problem is  symmetric, it is not restrictive to assume  that $y_1=-y_0$ and that  trajectories with the structure 1-$\H$-2 are like the ones described on Figure \ref{figexample2} with $\dis 0 \leq a \leq \frac{x_1-x_0}{2}$.

\begin{figure}[h]
\begin{center}
{ \resizebox{11cm}{!}{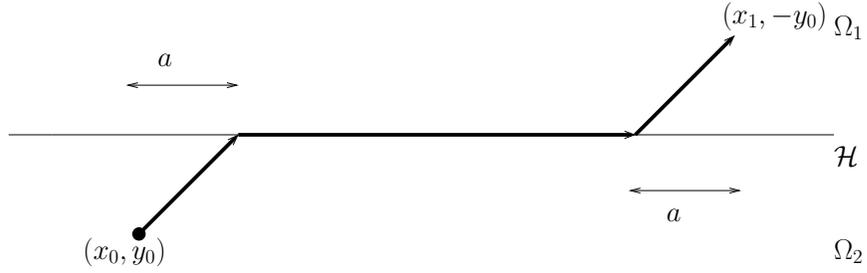} }
\end{center}
\caption{The trajectory 1-$\H$-2 .}\label{figexample2}
\end{figure}

For each trajectory steering $(x_0,y_0)$ to $(x_1,-y_0)$ a simple  computation gives the  cost 
(as a function of the parameter $a$)
$$
 C(a)=2 \sqrt{y_0^2+a^2} + \frac{x_1-x_0}{10}- \frac{a}{5}.
$$
Therefore, the value function is
$$
U(\xo,0;\xf)= \min_{0 \leq a \leq \frac{x_1-x_0}{2}} \:  \apt 2 \sqrt{y_0^2+a^2} + \frac{x_1-x_0}{10}- \frac{a}{5}  \cht \:.   
$$
and we obtain that:
\begin{itemize}
\item if  $\dis  \: \frac{x_1-x_0}{2} > \frac{|y_0|}{3 \sqrt{11}}$ then the optimal  trajectory has the structure 1-$\H$-2 with $a=\frac{|y_0|}{3 \sqrt{11}}$ and the optimal final time is  $\dis t_f=\frac{19}{3\sqrt{11}}- \frac{x_1-x_0}{10}$.
\item if $\dis \:   \frac{x_1-x_0}{2} \leq  \frac{|y_0|}{3 \sqrt{11}}$ then the optimal trajectory has the structure 1-2 with  $\dis a=\frac{x_1-x_0}{2}$  and  the optimal final time is  $\dis t_f= 2 \sqrt{y_0^2+\frac{(x_1-x_0)^2}{4}}$ (see Figure \ref{figexample3}).
\end{itemize}

\begin{figure}[h]
\begin{center}
{ \resizebox{11cm}{!}{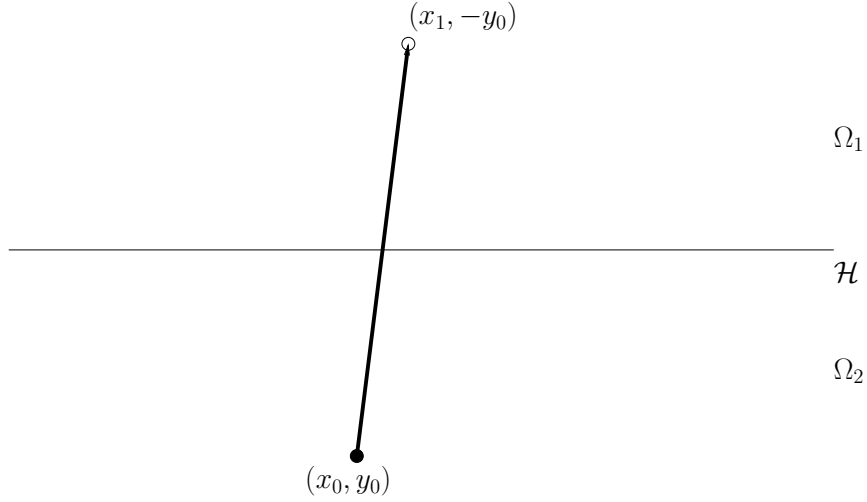} }
\end{center}
\caption{The trajectory 1-2.}\label{figexample3}
\end{figure}

We finally  remark  that, although this example is very simple, it is paradigmatic and illustrates many possible situations where one has two regions of the space (with specific dynamics) separated by an interface along which the dynamics are quicker than in the two regions. In this sense, the above example can be  adapted and complexified to represent some more realistic situations.

\section{Proofs}\label{sec_proofs}
\subsection{Proof of Proposition \ref{infvalues}}\label{proof_infvalues}
Fix $(\xo,\too;\xf,\tf) \in \overline{\Omega}_1 \times \R^+_* \times \overline{\Omega}_2 \times \R^+_*$,  with $\xo \neq \xf$ and
$t_1,t_2$ such that 
 $\tf > t_2 > t_1 > \too$.  Let $\vsx$ be the corresponding trajectory  solution of  \eqref{SysHyb}-\eqref{CondMixte}-\eqref{SCHstay}. 
 We construct three increasing $C^1$ diffeomorphisms:
$$ 
\rou : [\To,\Tu] \ds [\too,t_1] , \quad  \roh : [\To,\Tu] \ds [t_1,t_2] , \quad  \rod : [\To,\Tu] \ds [t_2,\tf]
$$
with $0 \leq \To < \Tu$ arbitrarily chosen. We solve then  the state equation \eqref{SysHybRip}  with controls 
$$
 \wu(\tau)=\rou^\prime(\tau) , \quad  \wdu(\tau)=\rod^\prime(\tau),     \quad  \wh(\tau)=\roh^\prime(\tau) ,
$$ 
$$
 \vu(\tau)=\alpha_1(\rou(\tau))=\alpha_1(t) , \quad  \vh(\tau)=\a_\H(\roh(\tau))=\a_\H(t), \:  \quad  \vd(\tau)=\alpha_2(\rod(\tau))=\alpha_2(t) ,
$$
and  initial and final data
$$
\Zob=(\xo,\too, X_1(t_1),t_1, X_2(t_2),t_2) , \: \Zub=(X_1(t_1),t_1,X_2(t_2),t_2,\xf,\tf )  \ . 
$$
Therefore $(\Zob,\Zub) \in \Sm(\xo,\too;\xf,\tf)$ and the duplicated trajectory  is such that   
 \begin{equation*}  \label{UgTray}
 Y_1(\tau)=X_1(\rou(\tau))=X_1(t) , \quad Y_\H(\tau)=X_\H(\roh(\tau))=X_\H(t) , \quad
 Y_2(\tau)=X_2(\rod(\tau))=X_2(t) ,
 \end{equation*}
 for any $t \in (\too,\tf)$, $\tau \in (\To,\Tu)$.
 Moreover, by  the above change of time variable we have 
 \begin{multline*} 
\dis \int_{\too}^{t_1} l_1(X_1(t),\alpha_1(t))  \: dt  + \int_{t_1}^{t_2} l_\H(X_\H(t),\a(t))  \:  dt +\int_{t_2}^{\tf} l_2(X_2(t),\alpha_2(t))   \: dt \\
=\int_{\To}^{\Tu} \Big(  l_1(Y_1(\tau),\vu(\tau))\wu(\tau)   +l_\H(Y_\H(\tau),\vh(\tau))\wh(\tau)  
+ l_2(Y_2(\tau),\vd(\tau))\wdu(\tau) \Big) \: d\tau . 
\end{multline*} 
Hence $ \co(\xo,\too;\xf,\tf; \vsx)=\cor(\vsz)$.  Conversely, since the time change of variable $\rho$  is invertible 
given 
$(\Zo,\Zu) \in \Sm(\xo,\too;\xf,\tf)$ 
and a corresponding admissible trajectory 
$\vsz$ we can construct  a trajectory  $\vsx$ such that $\cor(\vsz)=\co(\xo,\too;\xf,\tf; \vsx)$ and the proof is completed.

\subsection{Proof of Lemma \ref{PMPripa}}\label{proof_PMPripa}
The result follows by applying the usual Pontryagin maximum principle (see \cite{Po}).   
If $\Zo,  \Zu  \in \Sm(\xo,\too;\xf,\tf)$,  the classical transversality  condition  holds  (see \cite[Theorem 12.15]{AgSa} or \cite{Tlibro}):
\begin{equation*} 
(-\vapz(T_0), \vapz(T_1)) \: \:   \bot  \: \: T_{ ( \vsz(T_0), \vsz(T_1))} \:   \: \Sm(\xo,\too;\xf,\tf) .
\end{equation*}
Now, if $\vsz(T_0)=\Zo=(\xo,\too,x_1,t_1,x_2,t_2)$ and $\vsz(T_1)=\Zu=(x_1,t_1,x_2,t_2,\xf,\tf)$
 the above relation gives: 
 \begin{itemize}
 \item $t_1=Z_0^4=Z_1^2$  implies  $\varh(T_0)=\varu(T_1)$;
\item $t_2=Z_0^6=Z_1^4$ implies $\vard(T_0)= \varh(T_1)$;   
\item $x_1=Z_0^3=Z_1^1$ and $x_1 \in \H$ imply  $\vayh(T_0)=\vayu(T_1)+\nu_1 \:  \nabla \Psi(x_1)$;
\item $x_2=Z_0^5=Z_1^3$ and $x_2 \in \H$ imply $\vayd(T_0)= \vayh(T_1)+ \nu_2  \:  \nabla \Psi(x_2)$.
\end{itemize}
The result follows.

\subsection{Proof of Proposition \ref{TEOLINKBASE}}
 To apply  the classical theory of viscosity solutions for  Hamilton-Jacobi equations, we define   
 two different value functions by 
considering  separately the case when we fix  the initial  data $(\Zo,\To)$ and we consider a function the final data $(\Zu,\Tu)$ or conversely. 
 Precisely, to prove $(i)$ we fix  $(T_1,\Zu)$ and   for any  
$\To \in  \R_*^+$,  $\Zo  \in \overline{\Omega}_1 \times \R_*^+ \times \H \times \R_*^+ \times \overline{\Omega}_2 \times \R_*^+$  
we define $\Vr^0(\Zo,\To)= \Vr(\Zo,\To,\Zu,\Tu)$.  \\
Similarly to prove $(ii)$,   $(T_0,\Zo)$ is given and  for any  
$\Tu \in  \R_*^+$, $\Zu  \in  \overline{\Omega}_1 \times \R_*^+ \times \H \times \R_*^+ \times \overline{\Omega}_2 \times \R_*^+$  we 
set  $\Vr^1(\Zu,\Tu)=\Vr(\Zo,\To,\Zu,\Tu)$.

In order to write the partial derivatives of $\Vr^0$ and $\Vr^1$ we consider  a generic function 
  $u(\vsz,t):  \big(\Omega_1 \times \R_*^+ \times \H \times \R_*^+ \times \Omega_2 \times \R_*^+ \big)  \times \R^+  \rightarrow \R$ and 
we will enumerate the  variables  as follows $\vsz=(1,2,3,4,5,6)$. Therefore $\pt_i u(\vsz)$ takes values in 
$\R$ for $i=2,4,6$, $ \ps_i \Vr(\vsz) $  takes values in $\R^N$ for $i=1,5$ and   $ \nabla_{\H,i} \Vr $ in $T \H$ for $i=3$. 
We will set 
$$
\nabla  u = ( \ps_1 u ,\pt_2 u,\ps_{\H,3} u,\pt_4 u ,\ps_5 u ,\pt_6 u)  \qquad u_t (\vsz,t)= - \frac{\partial u}{ \partial t} (\vsz,t).
$$
Moreover, we  respectively denote by $D^{+} u $ and  $D^{-} u $  the classical super- and sub-differential in  the space variables $1,3,5$. 
We have 
$$
 \nabla_{\Zo}  \Vr(\Zo,\To; \Zu, \Tu) = \nabla \Vr^0 (\Zo,\To) \quad   \mbox{ and }  \quad \nabla_{\Zu}  \Vr(\Zo,\To; \Zu, \Tu) = \nabla \Vr^1 (\Zu,\Tu) . 
$$
By applying the standard theory of viscosity solution   (see, e.g., \cite[Propositions 3.1 and 3.5]{BCD}, see also \cite{Ba})
we know that  $\Vr^0(\Zo,\To)$ is  a bounded, Lipschitz continuous viscosity solution of 
\begin{equation*} 
- \frac{\partial u}{ \partial t} (\vsz,t)+ \Hr\Big(\vsz , -\nabla u, -1 \Big)=0  \quad \mbox{ in }(\Omega_1 \times \R_*^+ \times \H \times \R_*^+ \times \Omega_2  
\times \R_*^+) \times  (\To,\Tu)  ,
\end{equation*}
and   $\Vr^1(\Zu,\Tu)$ is a bounded, Lipschitz continuous  viscosity solution of 
\begin{equation*} 
 \frac{\partial u}{ \partial t} (\vsz,t)+ \Hr\Big(\vsz ,  \nabla u , -1 \Big)=0  \quad \mbox{  in } (\Omega_1 \times \R_*^+ \times \H \times \R_*^+ \times \Omega_2  
\times \R_*^+) \times (\To,\Tu) . 
\end{equation*}
Therefore, we can  apply  \cite[Corollary 3.45]{BCD} to obtain \eqref{linkbase0} and \eqref{linkbase1}.
 Now, if assumption (Hu) holds,  one can prove that the two functions  
$\tau \mapsto  \Vr^0(\vsz(\tau),\tau)$ and $\tau \mapsto \Vr^1(\vsz(\tau),\tau)$
 are differentiable    (see \cite[Theorem 7.4.16 ]{CannarsaSinestrari}  or 
 \cite{AubinFrankowska,ClarkeVinter})
 thus  \eqref{linkbaseT0} and   \eqref{linkbaseT1}  follow.

\subsection{Proof of Theorem \ref{mainthm}}\label{proof_mainthm}
Fix  $(\xo,\too;\xf,\tf)  \in \Omega_1 \times \R^+ \times \Omega_2 \times \R^+_* $.  
 To obtain the first result   we rewrite the equality \eqref{tesilem} of Proposition \ref{infvalues} as  
\begin{multline} \label{tesiVVnot}
\V_{1,\H,2}(\xofo,\xoff)= \inf  \Big\{  \Vr((\xofo,\xud);(\xud,\xoff))    \: : \:   \xud=(x_1,t_1,x_2,t_2), \  \Psi(x_1)=0 , \Psi(x_2)=0 ,  \\ \tf> t_2 > t_1 > \too  \Big\}
\end{multline}
where we  set   $\xof=(\xofo,\xoff)=(\xo,\too;\xf,\tf)$ and  $\xud=(x_1,t_1,x_2,t_2)$.  Thus, by the construction of the duplicated value function $\Vr$ we have  
\begin{multline*} 
\V_{1,\H,2}(\xo,\too;\xf,\tf)= \min \big\{  \V_{1}(\xo,\too;,x_1, t_1)+\V_{\H}(x_1, t_1; x_2, t_2) + \V_{2}(x_2,t_2; \xf,\tf) \ :\   \\   t_0 < t_1 < t_2< \tf , \  x_1 , x_2 \in \H  \big\} .
\end{multline*}
Thanks to  \eqref{tesilem}  in Proposition \ref{infvalues}  we can   consider now $(\Zo,\Zu) \in \Sm(\xo,\too;\xf,\tf)$ such that $\V_{1,\H,2}(\xo,\too;\xf,\tf)=\Vr(\Zo,\To;\Zu,\Tu)$ for an optimal trajectory 
$\vsz(\cdot) \in {\cal  Z}_{(\Zo,\To),(\Zu,\Tu)}$  (note that we have $\V_{1,\H,2}(\xo,\too;\xf,\tf)=\Vr(\Zo,\To;\Zu,\Tu)=C(\vsz) $). 
Let $\vapz(\cdot)$ be the adjoint vector given by Theorem \ref{PMPripa}, the maximality condition \eqref{condmax} implies that   
\begin{eqnarray} \label{jumpresultripa1P}
\psg{ \vayu(\tau)}{f_1(Y_1(\tau),\vu(\tau))} + \varu(\tau)  \:  +  p^0  \: l_1(Y_1(\tau),\vu(\tau)) & =  0    \\ 
 \label{jumpresultripa2P}
\psgH{ \vayh(\tau)}{f_\H(Y_\H(\tau),\vh(\tau))} + \varh(\tau)  \: +  p^0  \: l_\H(Y_\H(\tau),\vh(\tau))  & =  0    \\ 
 \label{jumpresultripa3P}
\psg{\vayd(\tau)}{ f_2(Y_2(\tau),\vd(\tau))} + \vard(\tau)  \:  +  p^0 \: l_2(Y_2(\tau),\vd(\tau)) & =  0  
\end{eqnarray}
for almost every $\tau \in (T_0,T_1)$.
Moreovever, by the transversality condition   in Theorem \ref{PMPripa},   there  exist $\nu_1, \nu_2 \in \R$ such that 
\begin{eqnarray}  
\label{Jumpt1dim}
\varh(T_0)&=&\varu(T_1) \\
\label{Jumpt2dim}
\vard(T_0)&=& \varh(T_1) \\
\label{Jumps1dim} 
\vayh(T_0)&=&\vayu(T_1)+\nu_1 \:  \nabla \Psi(x_1) \\
\label{Jumps2dim} 
\vayd(T_0)&=& \vayh(T_1)+ \nu_2  \:  \nabla \Psi(x_2) .
\end{eqnarray}
Our aim is now to interpret  these equalities on the original problem. 
 By  definition of the duplicated problem, we construct an optimal  trajectory  $\vsx(\cdot)$ for 
 $\V_{1,\H,2}(\xo,\too;\xf,\tf)$,
  such that 
 \begin{eqnarray*} 
   Y_1(\tau)=X_1(\rou(\tau))=X_1(t) &  \forall \tau \in (\To,\Tu) \quad \forall t \in (\too,t_1) \\
   Y_\H(\tau)=X_\H(\roh(\tau))=X_\H(t) & \forall \tau \in (\To,\Tu) \quad \forall t \in (t_1,t_2) \\
   Y_2(\tau)=X_2(\rod(\tau))=X_2(t) & \forall \tau \in (\To,\Tu) \quad \forall t \in (t_2,\tf) 
 \end{eqnarray*}
Indeed, we recall that by construction $C(\vsx)=C(\vsz)= \V_{1,\H,2}(\xo,\too;\xf,\tf)=\Vr(\Zo,\To;\Zu,\Tu)$.
We set now
 \begin{eqnarray*} 
 \vaxu(t)=\vayu(\tau)= \vayu(\rou(t))  &  \forall \tau \in (\To,\Tu) \quad \forall t \in (\too,t_1) \\
 \vaxh(t)= \vayh(\tau)= \vayh(\roh(t))  & \forall \tau \in (\To,\Tu) \quad \forall t \in (t_1,t_2) \\
   \vaxd(t)= \vard(\tau)= \vard(\rod(\tau)) & \forall \tau \in (\To,\Tu) \quad \forall t \in (t_2,\tf) .
 \end{eqnarray*}
 Therefore,  by  definition of the Hamitonians  $\Hus$, $\Hhs$ and $\Hds$, the equalities  \eqref{jumpresultripa1P}-\eqref{jumpresultripa3P}  give  
 \begin{eqnarray*} 
\Hus(X_1(t),\vaxu(t),p^0,\alpha_1(t))=-  \varu(\tau)  \\
\Hhs(X_\H(t),\vaxh(t),p^0,\a_\H(t))=-  \varh(\tau)  \\
\Hds(X_2(t),\vaxd(t),p^0,\alpha_2(t))=-  \vard(\tau).
\end{eqnarray*}
for almost every  $t \in (\too,\tf)$, $\tau \in (T_0,T_1)$. \\
To obtain the continuity conditions on the Hamiltonians we consider the above  equalities at   times $t_1$, $t_2$. By construction of the time change of variable and the continuity of the adjoint vector we have
\begin{multline*}
\dis  \Hus(X_1(t_1^-),\vaxu(t_1^-),p^0,\alpha_1(t_1^-)) =\lim_{t \ds t_1 ^-}  \Hus(X_1(t),\vaxu(t),p^0,\alpha_1(t)) =  \lim_{\tau \ds \Tu}  (-  \varu(\tau))= -\varu(T_1)
\end{multline*}
\begin{multline*}
\dis  \Hhs(X_\H(t_1^+),\vaxh(t_1^+),p^0, a_\H(t_1^+)) =\lim_{t \ds t_1 ^+}  \Hhs(X_\H(t),\vaxh(t),p^0,a_\H(t)) \\   =  \lim_{\tau  \ds \To}  (-  \varh(\tau))= -\varh(T_0)
\end{multline*}
\begin{multline*}
\dis  \Hhs(X_\H(t_2^-),\vaxh(t_2^-),p^0,a_\H(t_2^-))=\lim_{t \ds t_2 ^-}  \Hhs(X_\H(t),\vaxh(t),p^0,a_\H(t)) \\   = \lim_{\tau \ds \Tu}  (-  \varh(\tau))=-\varh(T_1)
\end{multline*}
\begin{multline*}
\dis  \Hds(X_2(t_2^+),\vaxd(t_2^+),p^0,\alpha_2(t_2^+)) =\lim_{t \ds t_2 ^+}  \Hds(X_2(t),\vaxd(t),p^0,\alpha_2(t))=  \lim_{\tau  \ds \To}  (-  \vard(\tau))= -\vard(T_0) . 
\end{multline*}
Since by  \eqref{Jumpt1dim}, \eqref{Jumpt2dim} we have  $\varh(T_0)=\varu(T_1)$ and $\varh(T_1)=\vard(T_0)$,
the above equalities give 
\begin{equation*} 
\Hus(X_1(t_1^-),\vaxu(t_1^-),p^0,\alpha_1(t_1^-))=\Hhs(X_\H(t_1^+),\vaxh(t_1^+),p^0,a_\H(t_1^+)) \:  
\end{equation*}
 \begin{equation*} 
 \Hhs(X_\H(t_2^-),\vaxh(t_2^-),p^0,a_\H(t_2^-))=\Hds(X_2(t_2^+),\vaxd(t_2^+),p^0,\alpha_2(t_2^+)) , 
\end{equation*}
therefore,   by the optimality of the trajectory, we can conclude that 
\begin{equation} \label{ContTudim}
\Hu(X_1(t_1^-),\vaxu(t_1^-),p^0)=\Hh(X_\H(t_1^+),\vaxh(t_1^+),p^0)
\end{equation}
\begin{equation} \label{ContTddim}
\Hh(X_\H(t_2^-),\vaxh(t_2^-),p^0)=\Hd(X_2(t_2^+),\vaxd(t_2^+),p^0) .
\end{equation}
To obtain the jump conditions on the adjoint vector we exploit the transversality conditions on the duplicated problem (\eqref{Jumps1} and \eqref{Jumps2} in  Theorem \ref{PMPripa}).  By  applying  the usual change  of variable  in 
\eqref{Jumps1dim}  and \eqref{Jumps2dim}
we have  
\begin{equation} \label{ConcontpOP}
\vaxh(t_1^+)=\vaxu(t_1^-)+ \nu_1 \: \nabla \Psi(X_1(t_1^-)) \quad \mbox{ and } \quad \vaxd(t_2^+)=\vaxh(t_2^-)+ \nu_2 \: \nabla \Psi(X_2(t_2^+)) .
\end{equation}
Note now that by definition of $\Hus$, $\Hds$, $\Hhs$ the continuity conditions \eqref{ContTudim}-\eqref{ContTddim} read
\begin{equation} \label{ContHt1exp} 
\psg{\vaxu(t_1^-)}{f_1(t_1^-)} + p^0 \: l_1(t_1^-) = \psgH{\vaxh(t_1^+)}{f_\H(t_1^+)} + p^0 \: l_\H(t_1^+) 
\end{equation}
\begin{equation} \label{ContHt2exp} 
\psgH{\vaxh(t_2^-)}{f_\H(t_2^-)} + p^0 \: l_\H(t_2^-) =\psg{\vaxd(t_2^+)}{f_2(t_2^+)}+ p^0 \: l_2(t_2^+) 
\end{equation}
where we used the  short notations $f_i(t_i^\pm)=f_i(X_i(t_i^\pm),\alpha_i(t_i^\pm))$ and 
$l_i(t_i^\pm)=l_i(X_i(t_i^\pm),\alpha_i(t_i^\pm))$ with  $i\in\{1,2,\H\}$. 
By using twice $\vaxh(t_1^+)=\vaxu(t_1^-)+ \nu_1 \: \nabla \Psi(X_1(t_1^-))$ and by recalling that by construction  
$ \psgH{\nabla \Psi(X_1(t_1^-))}{f_\H(t_1^+)}=0$ the equality  \eqref{ContHt1exp}  becomes
$$
 \nu_1\:  \psg{\nabla \Psi(X_1(t_1^-))}{f_1(t_1^-)}= \psg{\vaxh(t_1^+)}{f_1(t_1^-)} -\psgH{\vaxu(t_1^-)}{f_\H(t_1^+)}+ p^0 \: ( l_1(t_1^-)- l_\H(t_1^+) ) 
 $$
thus
$$
 \nu_1 = \frac{\psg{\vaxh(t_1^+)}{f_1(t_1^-)} -\psgH{\vaxu(t_1^-)}{f_\H(t_1^+)}+ p^0 \: ( l_1(t_1^-)- l_\H(t_1^+) ) }{\psg{\nabla \Psi(X_1(t_1^-))}{f_1(t_1^-)}} ,
$$
since by assumption  $\psg{\nabla \Psi(X_1(t_1^-))}{f_1(t_1^-)} \neq 0$.
Similarly, if we replace $\vaxd(t_2^+)=\vaxh(t_2^-)+ \nu_2 \: \nabla \Psi(X_2(t_2^+))$ in \eqref{ContHt2exp}  we obtain 
$$
\psgH{\vaxh(t_2^-)}{f_\H(t_2^-)} + p^0 \: l_\H(t_2^-) =\psg{\vaxh(t_2^-)+ \nu_2 \: \nabla \Psi(X_2(t_2^+))}{f_2(t_2^+)}+ p^0 \: l_2(t_2^+) .
$$
Thus 
\begin{equation*}
\nu_2 = \frac{\psgH{\vaxh(t_2^-)}{f_\H(t_2^-)}- \psg{\vaxh(t_2^-)}{f_2(t_2^+)}+ p^0 \: ( l_\H(t_2^-)-l_2(t_2^+))}{\psg{ \nabla \Psi(X_2(t_2^+))}{f_2(t_2^+)}} \: 
\end{equation*}
thanks to the assumption $\psg{ \nabla \Psi(X_2(t_2^+))}{f_2(t_2^+)} \neq 0$.

In order to conclude the proof we need, roughly  speaking, to  replace $P_1$, $P_2$ $P_\H$ by $Q_1$, $Q_2$ and $Q_\H$. 
To this aim  we  compute the  relation between $P_1$, $P_2$ $P_\H$ and the derivatives of $S_{1,\H,2}$. This is done  in Lemma   \ref{ugderVV} hereafter.

 \begin{lemma} \label{ugderVV}  
 Under the assumptions  (H$\H$), (Hfl$_i$), (Hfl$_\H$)  and (Hu), given   $(\xo,\too;\xf,\tf) \in \Omega_1 \times \R^+_* \times \Omega_2 \times \R^+_*$, if $\xud=(x_1,t_1,x_2,t_2)=\xud((\xo,\too;\xf,\tf))$ is a minimum point in \eqref{tesiVVnot}, then
\begin{equation} \label{ugderVV1}
\frac{\partial  }{\partial \too} \: \V_{1,\H,2} \big( \xo,\too;\xf,\tf\big)= \pt_2 \Vr \big( (\xo,\too,\xud),(\xud,\xf,\tf)   \big) 
\end{equation}  
 \begin{equation*} 
\frac{\partial  }{\partial \tf} \: \V_{1,\H,2}  \big( \xo,\too;\xf,\tf\big) = \pt_{12} \Vr \big((\xo,\too,\xud),(\xud,\xf,\tf)\big) 
\end{equation*} 
  \begin{equation} \label{ugderVV3}
\nabla_{\xo} \: \V_{1,\H,2}  \big( \xo,\too;\xf,\tf \big) = \ps_{1} \Vr \big(  (\xo,\too,\xud),(\xud,\xf,\tf)   \big) 
\end{equation}  
  \begin{equation*} 
\nabla_{\xf} \: \V_{1,\H,2}  \big(  \xo,\too;\xf,\tf  \big) = \ps_{11} \Vr \big(  (\xo,\too,\xud),(\xud,\xf,\tf)  \big) .
\end{equation*}  
Moreover, if $(\xo,\too;\xf,\tf) \in \Omega_1 \times \R^+_* \times \H \times \R^+_*$ then 
\begin{equation*} 
\nabla^\H_{\xf} \: \V_{1,\H,2}  \big(  \xo,\too;\xf,\tf  \big) = \ps_{9}^\H \Vr \big(  (\xo,\too,\xud),(\xud,\xf,\tf)  \big) 
\end{equation*}
and if $(\xo,\too;\xf,\tf) \in \H \times \R^+_* \times \Omega_2 \times \R^+_*$ then 
\begin{equation*}
\nabla^\H_{\xo} \: \V_{1,\H,2}  \big(  \xo,\too;\xf,\tf  \big) = \ps_{3}^\H \Vr \big(  (\xo,\too,\xud),(\xud,\xf,\tf)  \big) .
\end{equation*}
\end{lemma}

Before proving this lemma, let us conclude the proof. By \eqref{ugderVV3} in Lemma  \ref{ugderVV} we have 
\begin{equation*} 
\nabla_{\xo} \: \V_{1,\H,2}  \big( X_1(t), t ;\xf,\tf \big) = \ps_{1} \Vr \big(  \vsz(\tau),\tau ; \Zu, \Tu   \big)  \quad \forall \tau \in (\To,\Tu) \quad \forall t \in (\too,t_1), 
\end{equation*}  
therefore, by the continuity of the adjoint vector and  \eqref{linkbase1-7},  we have 
\begin{multline*}
 P_1( t_1^-)=\lim_{t \ds t_1^-}  P_1(t)= \lim_{\tau \ds \Tu} \vayu (\tau) \\
 =\lim_{\tau \ds \Tu} -\ps_{1} \Vr \big(  \vsz(\tau),\tau ; \Zu, \Tu   \big)=
 \lim_{t \ds t_1^- } - \nabla_{\xo} \: \V_{1,\H,2}  \big( X_1(t), t ;\xf,\tf \big) ,
\end{multline*}
that is, $P_1(t_1^-)=Q_1(t_1^-) $.  \\
In a  similar way, by   Lemma  \ref{ugderVV} below,  equalities \eqref{linkbase3-9}-\eqref{linkbase5-11}
and the continuity of the adjoint vector, we obtain $P_2(t_2^+)=Q_2(t_2^+)$, $P_\H(t_2^-)=Q_\H(t_2^-)$. This  concludes the proof of Theorem \ref{mainthm}.

\begin{proof}[Proof of Lemma \ref{ugderVV}.]  
Given $\xof=(\xo,\too;\xf,\tf)$,  let  $\xud=(x_1,t_1,x_2,t_2)=(x_1(\xof), t_1(\xof), x_2(\xof), t_2(\xof))$ be a minimum point in \eqref{tesiVVnot}. We can then  write
$$
\V_{1,\H,2} (\xo,\too;\xf,\tf)=\Vr\Big( \big( \xo,\too,x_1(\xof), t_1(\xof), x_2(\xof), t_2(\xof) \big) \: ;  \big(x_1(\xof), t_1(\xof), x_2(\xof), t_2(\xof),\xf, \tf \big)\Big) .
$$
We first remark that putting together \eqref{Jumpt1}-\eqref{Jumps2} in  Theorem \ref{PMPripa} and \eqref{linkbaseT01} in Remark 
\ref{remHJtotu} we have 
\begin{equation} \label{CN}
\begin{array}{rcl}
(\pt_4 \Vr + \pt_8 \Vr)\big( (\xofo,\xud(\lambda));(\xud(\lambda),\xoff) \big)& = &0  \\
(\pt_6 \Vr + \pt_{10} \Vr)\big( (\xofo,\xud(\lambda));(\xud(\lambda),\xoff) \big)& = &0  \\
(\ps_3 \Vr + \ps_{7} \Vr) \big( (\xofo,\xud(\lambda));(\xud(\lambda),\xoff) \big)&= &\nu_1 \nabla \Psi (x_1)  \\
(\ps_5 \Vr + \ps_{9} \Vr) \big( (\xofo,\xud(\lambda));(\xud(\lambda),\xoff) \big)& = &\nu_2 \nabla \Psi (x_2) \\
\Psi(x_1(\lambda))&= &0 \\
\Psi(x_2(\lambda))&=& 0 .
\end{array}
\end{equation}
We will only detail the proof of \eqref{ugderVV1}   and  \eqref{ugderVV3}, the other proofs being  similar.  
If we  set  $\bar{\chi}=\big( (\xofo,\xud(\xof)),(\xud(\xof),\xoff)\big)$ by simple computations we get 
\begin{multline*}
\dis \frac{\partial \V }{\partial \too}  (\xof) = \dis  
\pt_2 \Vr(\bar{\chi})+ \dis
 \langle \ps_3 \Vr(\bar{\chi}),  \frac{\partial x_1}{ \partial \too}(\xof) \rangle+ 
 \pt_4 \Vr(\bar{\chi})\:\frac{\partial t_1}{ \partial \too}(\xof)+
  \langle \ps_5 \Vr(\bar{\chi}),  \frac{\partial x_2}{ \partial \too}(\xof) \rangle 
  +\pt_6 \Vr(\bar{\chi})\:\frac{\partial t_2}{ \partial \too}(\xof)\\
  +   \langle \ps_7 \Vr(\bar{\chi}),  \frac{\partial x_1}{ \partial \too}(\xof) \rangle+
   \pt_8 \Vr(\bar{\chi})\:\frac{\partial t_1}{ \partial \too}(\xof)+
  \langle \ps_9 \Vr(\bar{\chi}),  \frac{\partial x_2}{ \partial \too}(\xof) \rangle+
  \pt_{10} \Vr(\bar{\chi})\:\frac{\partial t_2}{ \partial \too}(\xof).
 \end{multline*}
Therefore, thanks to \eqref{CN}, we have
$$
\dis \frac{\partial \V }{\partial \too}  (\xof) =\pt_2 \Vr(\bar{\chi})+ \langle \mu_1 \nabla \Psi(x_1(\xof)),  \frac{\partial x_1}{ \partial \too}(\xof) \rangle+
  \langle \mu_2 \nabla \Psi(x_2(\xof)),  \frac{\partial x_2}{ \partial \too}(\xof) \rangle .
$$
Moreover, since  differentiating conditions $ \Psi(x_1(\xof))=0,\: \Psi(x_2(\xof))=0$ in \eqref{CN}  we obtain
$$
\langle  \nabla \Psi(x_1(\xof)),  \frac{\partial x_1}{ \partial \too}(\xof) \rangle=0  \quad \mbox{ and } 
\langle  \nabla \Psi(x_2(\xof)),  \frac{\partial x_2}{ \partial \too}(\xof) \rangle=0
$$ 
and we conclude that $\displaystyle \frac{\partial \V }{\partial \too}  (\xof) =\pt_2 \Vr(\bar{\chi})$.

Similarly,
\begin{multline*}
\dis \frac{\partial \V }{\partial \xo}  (\xof) = \dis  
\ps_1 \Vr(\bar{\chi})+ \dis
 \langle \ps_3 \Vr(\bar{\chi}),  \frac{\partial x_1}{ \partial \xo}(\xof) \rangle+ 
 \pt_4 \Vr(\bar{\chi})\:\frac{\partial t_1}{ \partial \xo}(\xof)+
  \langle \ps_5 \Vr(\bar{\chi}),  \frac{\partial x_2}{ \partial \xo}(\xof) \rangle \\
  +\pt_6 \Vr(\bar{\chi})\:\frac{\partial t_2}{ \partial \xo}(\xof) 
+    \langle \ps_7 \Vr(\bar{\chi}),  \frac{\partial x_1}{ \partial \xo}(\xof) \rangle+
   \pt_8 \Vr(\bar{\chi})\:\frac{\partial t_1}{ \partial \xo}(\xof) \\
   +
  \langle \ps_9 \Vr(\bar{\chi}),  \frac{\partial x_2}{ \partial \xo}(\xof) \rangle+
  \pt_{10} \Vr(\bar{\chi})\:\frac{\partial t_2}{ \partial \xo}(\xof).
 \end{multline*}
Thanks to \eqref{CN}, this gives
$$
\dis \frac{\partial \V }{\partial \xo}  (\xof) =\ps_{1} \Vr(\bar{\chi})+ \langle \nu_1 \nabla \Psi(x_1(\xof)),  \frac{\partial x_1}{ \partial \xo}(\xof) \rangle+
  \langle \nu_2 \nabla \Psi(x_2(\xof)),  \frac{\partial x_2}{ \partial \xo}(\xof) \rangle
$$
  by differentiating conditions $ \Psi(x_1(\xof))=0,\: \Psi(x_2(\xof))=0$ in \eqref{CN} we have 
$$
\langle  \nabla \Psi(x_1(\xof)),  \frac{\partial x_1}{ \partial \xo}(\xof) \rangle=0  \quad \mbox{ and } 
\langle  \nabla \Psi(x_2(\xof)),  \frac{\partial x_2}{ \partial \xo}(\xof) \rangle=0
$$ 
and hence $\dis \frac{\partial \V }{\partial \xo}  (\xof) = \ps_{1}\Vr(\bar{\chi})$.
 \end{proof}


\end{document}